\documentclass[final,onefignum,onetabnum,11pt]{siamart250211}

\usepackage{amsmath,verbatim,amssymb,amsfonts,amscd, graphicx,graphics, physics}
\usepackage{enumitem}
\usepackage{placeins}
\usepackage{subcaption}
\usepackage{algorithm}
\usepackage{algorithmic}

\usepackage{pifont}% http://ctan.org/pkg/pifont
\usepackage{stfloats}
\usepackage{makecell}

\newsiamthm{assumption}{Assumption}
\newsiamthm{defin}{Definition}
\newcommand{\cmark}{\ding{51}}%

\newlength{\tempdima}
\newcolumntype{M}[1]{>{\centering\arraybackslash}m{#1}}

\title{
BISTRO - A Bi-Fidelity Stochastic Gradient Framework using Trust-Regions for Optimization Under Uncertainty
\thanks{Submitted to arXiv on December 9th, 2025.
\funding{This work was funded by the NASA Project: Entry Systems Modeling under grant no. 80NSSC22K1007.}}
}

\headers{Bi-fidelity Stochastic Optimization under Uncertainty}{T. O. Dixon, G. F. Bomarito, J. E. Warner, and A. A. Gorodetsky}

\author{Thomas O. Dixon \thanks{Department of Aerospace Engineering, University of Michigan, Ann Arbor, MI 
  (\email{tdixono@umich.edu}, \email{goroda@umich.edu}).}
\and Geoffrey F. Bomarito \thanks{Durability, Damage Tolerance, and Reliability Branch, NASA Langley Research Center, Hampton, VA  
  (\email{geoffrey.f.bomarito@nasa.gov},\email{james.e.warner@nasa.gov}).}
  \and James E. Warner \footnotemark[3]
\and Alex A. Gorodetsky \footnotemark[2] }

\begin{document}

\maketitle

\begin{abstract}
Stochastic optimization of engineering systems is often infeasible due to repeated evaluations of a computationally expensive, high-fidelity simulation.
Bi-fidelity methods mitigate this challenge by leveraging a cheaper, approximate model to accelerate convergence.
Most existing bi-fidelity approaches, however, exploit either design-space curvature or random-space correlation, not both.
We present BISTRO - a BI-fidelity Stochastic Trust-Region Optimizer for unconstrained optimization under uncertainty through a stochastic approximation procedure.
This approach exploits the curvature information of a low-fidelity objective function to converge within a basin of a local minimum of the high-fidelity model where low-fidelity curvature information is no longer valuable. The method then switches to a variance-reduced stochastic gradient descent procedure.
We provide convergence guarantees in expectation under certain regularity assumptions and ensure the best-case $\mathcal{O}(1/n)$ convergence rate for stochastic optimization.
On benchmark problems and a 20-dimensional space shuttle reentry case, BISTRO converges faster than adaptive sampling and variance reduction procedures and cuts computational expense by up to 29x.
\end{abstract}



\section{Introduction}

Simulation-based optimization is a fundamental component to engineering design and decision-making, especially when direct experimentation with complex systems is too costly or has high risk \cite{hannah2014semiconvex,glowinski1995simulation}.
Computational models provide a useful surrogate for the real systems that can be used to guide decisions.
Computational models of real-world systems, however, are inherently imprecise.
Modeling these systems often requires both uncertain parameters and inherent randomness to account for incomplete knowledge of the real world.

Applications ranging from financial markets \cite{yin2002recursive} to energy system control \cite{roald2023power} must optimize under uncertainty to produce robust and reliable designs that reduce risk. Risk metrics such as probabilities of failure, however, cannot be evaluated explicitly for nonlinear, black-box systems.
Sampling-based approaches are instead employed to quantify the risk. 

Sampling-based methods for optimization under uncertainty are largely grouped into either sample average approximation (SAA) \cite{fu2015handbook, homem2014stochastic, shapiro1991asymptotic} or stochastic approximation (SA) \cite{bottou2018optimization, fehrman2020convergence, pasupathy2018sampling}.
SAA is an approach that fixes the sample set used for estimation, converting a stochastic optimization problem into a deterministic one. Since the optimization solution depends on these fixed samples, the underlying optimal solution is a random variable whose variance depends on the number of samples drawn.
The SA approach instead draws new samples during optimization and converges to the underlying optimal solution.
This work addresses optimization under uncertainty with SA because the true solution is desired asymptotically.

Unlike SAA, the SA approach is inherently random, necessitating stochastic optimization procedures \cite{zhao2018data,qiu2015robust,menhorn2024multilevel,chaudhuri2020risk}.
When the objective function is a sampling-based risk estimate, the computer simulation of a complex system needs to be evaluated potentially hundreds or thousands of times during optimization \cite{alarie2021optimization, toscano2022bayesian}.
When the simulation involves a high-fidelity computer model, each evaluation of the simulation can be expensive, taking from hours to days to run.
With such an expensive simulation, stochastic optimization becomes infeasible.
To reduce the computational requirements of engineering design under uncertainty, the variance of the objective function or gradient estimates can be reduced with variance-reduction strategies \cite{schmidt2017minimizing,johnson2013accelerating,shashaani2018astro}. 

Alternately, cheaper and less accurate simulations have been used instead of the highest-fidelity simulation available to reduce computational requirements \cite{eldred2002formulations,acar2021modeling}.
While a solution may be found using this low-fidelity simulation, it is a sub-optimal solution with respect to the high-fidelity, expensive simulation.
To maintain high-fidelity accuracy while reducing computational cost, multi-fidelity methods have been proposed \cite{ha2025multi,de2020bi,menhorn2024multilevel}.
Multi-fidelity methods offer a promising path forward by strategically combining models of different computational costs and accuracies.
The central idea is to use cheap, lower-fidelity models to guide the expensive, high-fidelity optimization while maintaining solution quality.
The majority of existing multi-fidelity optimization under uncertainty methods typically exploit one of two key relationships: either the correlation structure in the random space (using variance reduction techniques like multi-level Monte Carlo) \cite{dereich2019general,frikha2016multi,menhorn2024multilevel,agrawal2023multi,ng2014multifidelity,de2020bi,ha2025multi} or the curvature/geometry information in the design space (using trust-region or surrogate-based approaches) \cite{korondi2021multi,shah2015multi,ha2025multi}.
Methods that only exploit correlation reduce the variance of stochastic estimates but cannot leverage low-fidelity curvature to take informed steps.
Conversely, methods that only exploit curvature can rapidly navigate the design space but cannot reduce the inherent noise in estimation.

We propose that multi-fidelity SA methods would benefit from exploiting both of these relationships, correlation and curvature.
We call our proposed algorithm BISTRO (BI-fidelity Stochastic Trust-Region Optimizer); it is a gradient-based algorithm to exploit both design-space curvature and random-space correlation across two fidelities.
BISTRO combines a bi-fidelity trust-region procedure with control variate-based gradient estimation, enabling rapid exploration of the design space while efficiently handling noise in stochastic estimates.
The algorithm operates in two phases: (1) an initial trust-region phase that uses low-fidelity curvature to quickly navigate to promising regions, and (2) a variance-reduced stochastic gradient descent phase \cite{dereich2019general} that ensures asymptotic convergence with optimal rates once in a basin near the local minimum. 

Our approach attempts to mitigate the issues found with previous multi-fidelity optimization under uncertainty algorithms.
Previous trust-region approaches have demonstrated that low-fidelity models can deteriorate the performance of optimization if the low-fidelity model does not have useful information to exploit \cite{march2012provably,babcock2024multi,ha2025multi}.
Specifically, the low-fidelity model is typically dissimilar near the high-fidelity solution since the low-fidelity usually does not share its optimum with the high-fidelity model.
Thus, the multi-fidelity trust-region algorithms may struggle to utilize the low-fidelity model near the high-fidelity solution, wasting resources on evaluating the low-fidelity model. 
BISTRO tackles this problem by terminating the use of the trust-region near the high-fidelity solution and switching to a variance-reduced gradient descent algorithm instead.
By dropping the trust-region approach near the high-fidelity solution, we ignore the potential curvature information and instead focus on variance reduction.
This tactic may significantly reduce the cost of optimization in two manners: 
\begin{enumerate}
    \item Ignoring correlation far from the optimal solution reduces the number of low-fidelity evaluations when the statistical error is not dominant. The curvature is instead exploited to rapidly approach the solution.
    \item Ignoring curvature near the optimal solution reduces the number of low-fidelity evaluations when the statistical error is dominant. The correlation is instead exploited to significantly reduce the estimator variance.
\end{enumerate}
BISTRO can also be interpreted as an adaptive warm-starting technique that uses the curvature of the low-fidelity model to find an appropriate starting point for the reduced-variance stochastic gradient approach.
The previous work by \cite{dereich2019general} introduces a multi-level stochastic gradient descent that we adaptively warm-start with the trust-region phase.
In this work, we contribute further convergence guarantees and auxiliary conditions that are required for faster convergence compared to single-fidelity stochastic gradient descent.

Further, BISTRO deviates from prior stochastic trust-region methods in terms of the per-iteration budget.
Previous SA trust-region approaches rely on accurate estimators to ensure that the statistical error is smaller than a sufficient decrease criteria \cite{kouri2013trust, chen2018stochastic, deng2009variable, ha2025multi, ha2024adaptive,van2021mg}.
These adaptive approaches increase the estimator accuracy by increasing the number of samples drawn per iteration - gradually increasing the per-iteration cost.
Unless the number of samples increases according to convergence-rate-dependent sampling rates, these adaptive approaches are susceptible to being inefficient, i.e., not achieving the fastest convergence rate per simulation evaluation \cite{pasupathy2018sampling}.
BISTRO avoids asymptotically increasing the per-iteration cost by appealing to stochastic gradient descent, which obtains the best-case accuracy rate of $\mathcal{O}(1/n)$ under regularity assumptions.
Because BISTRO does not appeal to increasing the per-iteration cost, the trust-region size also does not need to vanish for convergence. 
BISTRO uses a fixed-sized trust-region to ensure low-fidelity regularity assumptions and prohibit arbitrarily large step sizes.

Finally, while BISTRO is introduced as a bi-fidelity algorithm, it can be easily extended to an ensemble of models using the multi-level Monte Carlo estimator introduced in Section \ref{sec:UQ}.
Adding more low-fidelity models may further reduce the gradient estimator variance, leading to a faster asymptotic convergence of BISTRO.
Further, our proposed approach is able to minimize any estimator that is compatible with multi-level Monte Carlo estimation, but we only provide empirical results for expectation minimization.
Demonstrating minimization of other statistics, as well as constrained optimization, is left to future work.

Our key contributions include:
\begin{itemize}
    \item A novel algorithm that jointly exploits both curvature and correlation information across model fidelities, providing superior performance compared to methods that use only one type of information.
    \item Convergence guarantees to a stationary point and asymptotic convergence rates comparable to reduced-variance stochastic gradient descent while offering dramatically improved finite-sample performance.
    \item Theoretical conditions that demonstrate when the reduced-variance stochastic gradient descent converges faster than its single-fidelity counterpart.
    \item Comprehensive numerical validation on problems ranging from quadratic test functions to a 20-dimensional space shuttle reentry optimization, demonstrating up to 29-fold computational savings.
\end{itemize}

The remainder of this paper is organized as follows.
Section \ref{sec:background} provides background on multi-fidelity optimization, uncertainty quantification, and optimization under uncertainty.
Section \ref{sec:MFOuU} provides related works that tackle multi-fidelity optimization under uncertainty.
Section \ref{sec:method} introduces the BISTRO algorithm and describes its key components.
Section \ref{sec:theo_results} presents our theoretical convergence analysis.
Finally, Section \ref{sec:results} demonstrates the effectiveness of BISTRO through numerical experiments on increasingly complex test problems.

\section{Background}\label{sec:background}

A background is provided on multi-fidelity optimization in Section \ref{sec:MFO}, multi-fidelity uncertainty quantification in Section \ref{sec:UQ}, and optimization under uncertainty in Section \ref{sec:OoU}.

\subsection{Multi-fidelity optimization}\label{sec:MFO}

Consider the unconstrained optimization problem
\begin{align}
    x^* = \arg\min_{x\in \mathbb{R}^d} ~J(x,f(x)),
\end{align}
where the objective function $J:\mathbb{R}^d \times \Theta \to\mathbb{R}$ depends on the value of a computationally expensive simulation $f(x):\mathbb{R}^d \to \Theta \subset \mathbb{R}^{m} $ at the fixed decision variable $x\in\mathbb{R}^d$.
Numerical optimization algorithms require many evaluations of $J$ and, by extension, many queries of the simulation $f(x)$. A vast majority of optimization research seeks to minimize the number of $J$ queries by discovering and exploiting problem structure \cite{nocedal2006numerical, boyd2004convex}.

Another, complementary approach to lower the cost of optimization is to exploit lower-fidelity, less-expensive simulation models in place of, or in conjunction with, $f$. These lower-fidelity simulations can include approaches based on partial differential equations with a coarser grids, reduced-physics simulations, or even  empirical models. Optimization approaches that entirely discard $f$ and are based solely on these low-fidelity simulations produce solutions that deviate from those of the high-fidelity model. However, {\it multi-fidelity optimization} approaches seek to achieve the best of both worlds: high accuracy and low computational expense.

Let $J_H(x)\equiv J(x,f(x))$ represent the original objective function using the high-fidelity simulation and $J_L(x)\equiv J(x,f_L(x))$ denote a lower-cost approximation. 
Multi-fidelity methods strategically combine the two fidelities to accelerate convergence to the high-fidelity solution by reducing the number of expensive function evaluations.
The shared information between the high- and low-fidelity simulations is exploited to inform the optimization procedure. Such  methods can be broadly categorized into global and local methods. Global methods create a surrogate of the high-fidelity objective function $J_H$ by using models of varying fidelity. To search the entire domain, global methods typically use a Bayesian optimization method with the created surrogate \cite{forrester2007multi,pellegrini2023multi}. While some global methods provide convergence guarantees to a global minimum \cite{do2023multi}, they tend to scale poorly with the number of optimization variables \cite{eriksson2021high, binois2022survey}. Local methods, on the other hand, exploit shared high- and low-fidelity information in the neighborhood of the current optimization step \cite{alexandrov1998trust,march2011gradient,march2012provably,eldred2004second,robinson2008surrogate,bryson2018multifidelity,hart2023hyper,pellegrini2022derivative,wu2022gradient}. These methods often provide only local convergence guarantees, requiring convexity to provide global convergence proofs \cite{rodriguez1998convergence}. An intuitive local approach to multi-fidelity optimization is to use the low-fidelity optimal solution as an initial guess for a high-fidelity solver, deemed as warm-starting \cite{hart2023hyper,pellegrini2022derivative,wu2022gradient}.
While warm-starting does not require new optimization algorithms and reduces the number of high-fidelity evaluations, it does not exploit the full utility of the low-fidelity model and solely relies on the high-fidelity model to converge.

Instead of warm-starting, this work focuses on trust-region-inspired approaches.
Trust region methods create a local, approximate model of $J_H$ to search over in a trusted region around the current optimization step \cite{alexandrov1998trust}. Multi-fidelity trust-region methods use the low-fidelity model as the local approximation and are able to produce high-fidelity solutions at a significantly reduced cost \cite{march2011gradient,march2012provably,eldred2004second,robinson2008surrogate,bryson2018multifidelity,peherstorfer2018survey}.
Let $\{x_0,x_1,\cdots,x_k\}$ be the steps taken by an optimization procedure, such that as $k\to\infty$ the procedure converges $x_k\to x^*$ to the solution $x^*=\arg\min_x J_H(x)$. The trust-region approach minimizes a low-fidelity objective in a trusted region, centered at the current step $x_k$, to propose a new step $x_{k+1}$.
More specifically, the next iterate $x_{k+1}$ is the solution to the optimization problem
\begin{align}
    x_{k+1} = \arg\min_y ~J_L(y) \quad
    \text{subject to }\quad  ||y-x_{k}||\leq \Delta_{k} \nonumber
\end{align}
where $x_{k}$ is the center of the trust-region, and the size of the trust-region $\Delta_{k}$ grows or shrinks based on the similarity between $J_L$ and $J_H$.
Convergence to the high-fidelity solution also requires
\begin{align}
    J_H(x_{k}) = J_L(x_k) \quad\quad\textrm{and}\quad\quad\nabla J_H(x_{k}) = \nabla J_L(x_k) \label{eq:trust_requirements}
\end{align}
at the center of each trust region. This condition is often satisfied by updating the low-fidelity model, as discussed in Section \ref{sec:trust-region}.

Practically, if the low-fidelity objective $J_L$ is structurally similar to $J_H$, then most evaluations will use the cheaper simulator $f_L$. A more informative $J_L$ allows the trust-region size $\Delta_{k}$ to grow, limiting the number of high-fidelity evaluations while still converging to the high-fidelity solution.

\begin{figure}[t]
    \centering
    \includegraphics[width=0.49\linewidth,clip,trim={20pt 10pt 30pt 30pt}]{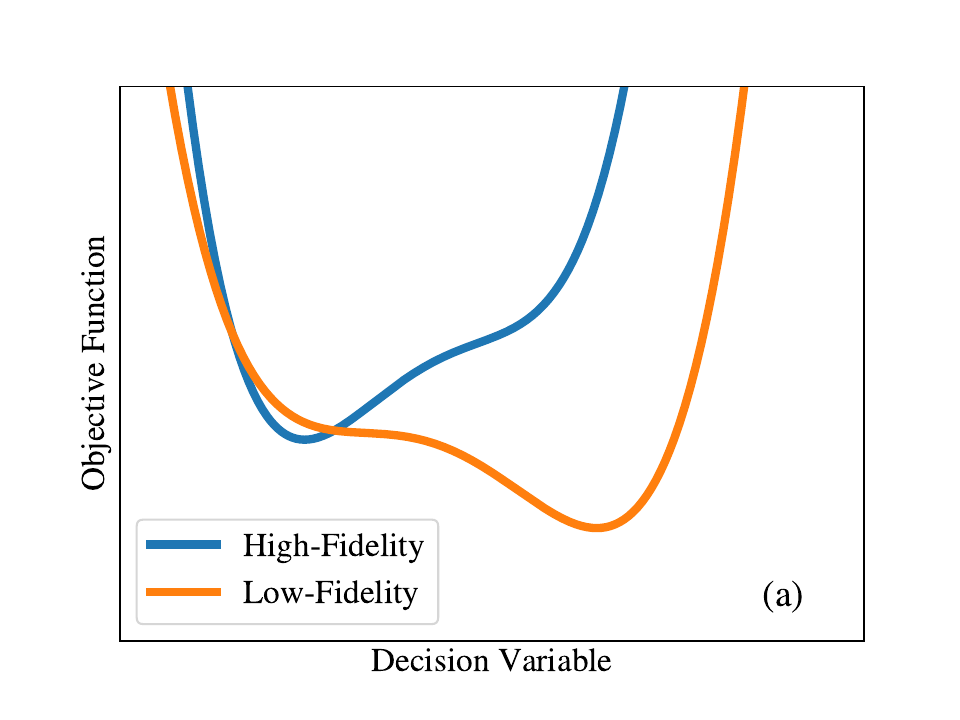}
    \includegraphics[width=0.49\linewidth,clip,trim={20pt 10pt 30pt 30pt}]{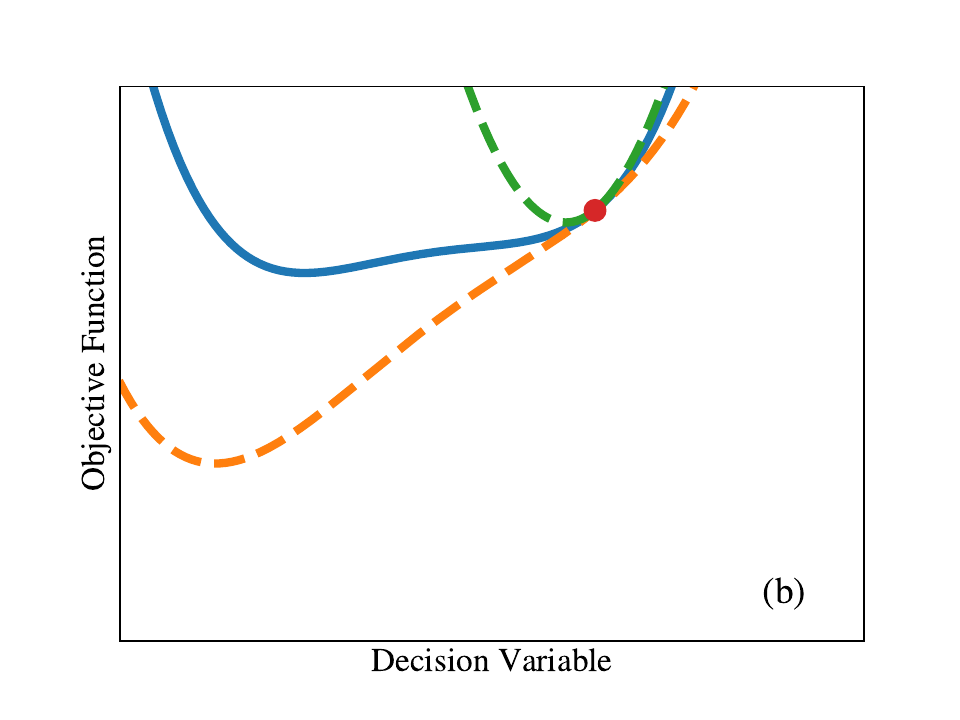}
    \includegraphics[width=0.49\linewidth,clip,trim={20pt 10pt 30pt 30pt}]{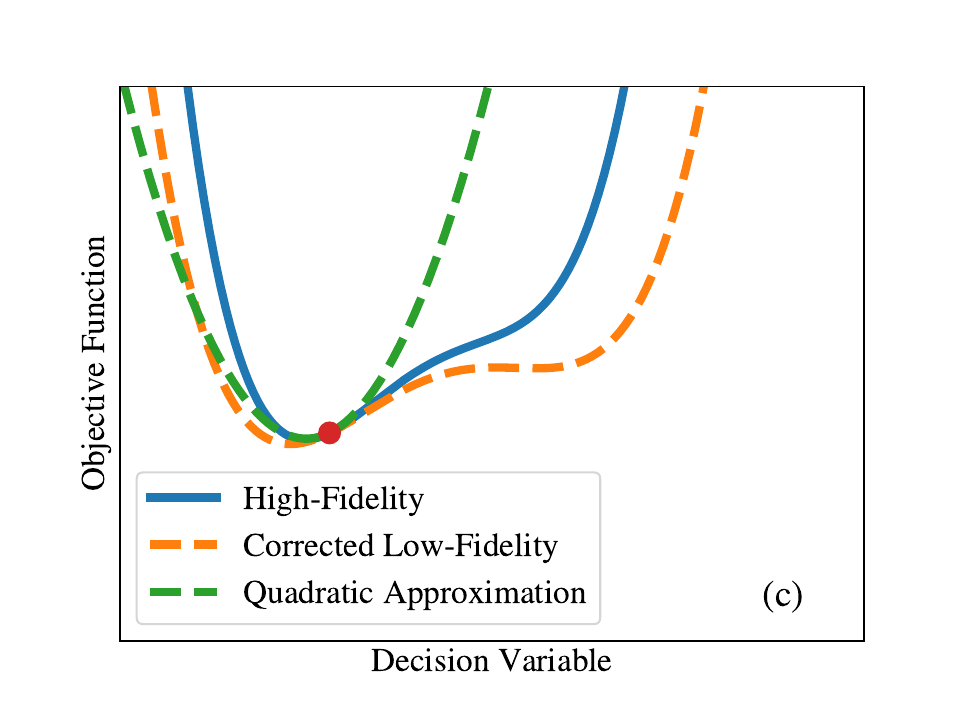}
    \includegraphics[width=0.49\linewidth,clip,trim={20pt 10pt 30pt 30pt}]{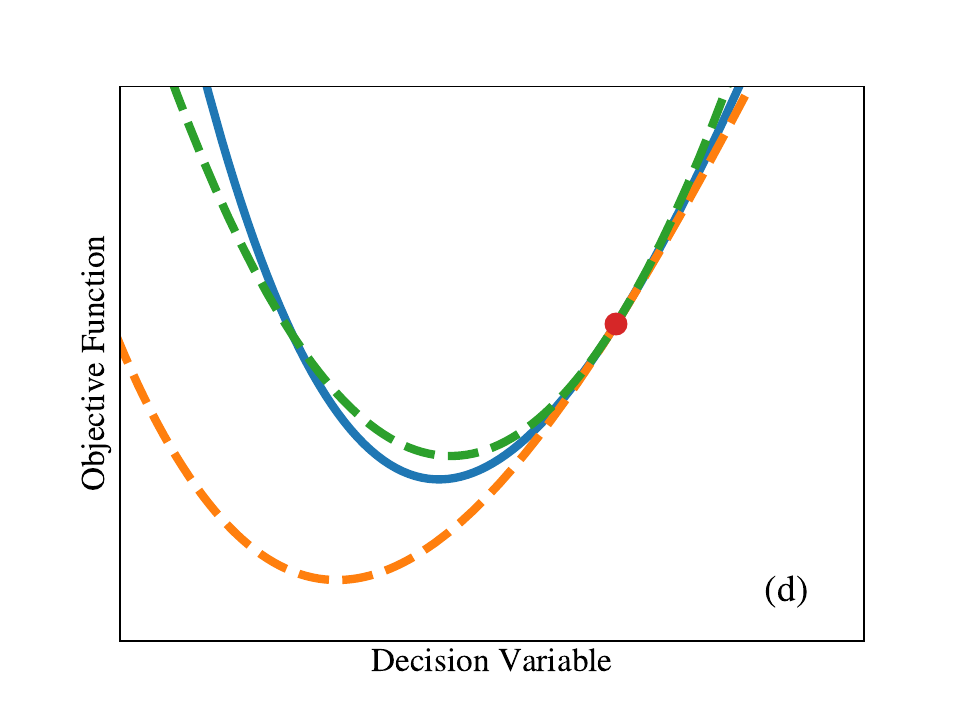}
    \caption{\textbf{Example} - (a) An example of high- and low-fidelity objective functions. (b) The corrected low-fidelity function and a quadratic approximation of the high-fidelity when corrected far from the high-fidelity solution. The corrected function's minimum is closer to the high-fidelity solution. (c) The corrected low-fidelity function and a quadratic approximation of the high-fidelity function when corrected near the high-fidelity solution.  (d) A close-up view of plot (c) near the solution. The quadratic approximation's minimum is closer to the high-fidelity solution. }
    \label{fig:example_system}
\end{figure}

One major flaw of multi-fidelity trust-region approaches is that the cost-benefit deteriorates when the low-fidelity is a poor approximation to the high-fidelity objective.
Previous works have shown that misleading low-fidelity models can slow the convergence of multi-fidelity trust-region algorithms, even performing worse than the single-fidelity counterparts \cite{march2012provably, babcock2024multi,ha2025multi}. 
Unless the low-fidelity objective shares the minimum of the high-fidelity objective (it does not in practice), the low-fidelity objective simply incurs additional cost to the optimizer without providing useful information.
Figure \ref{fig:example_system} displays an example system that demonstrates this tradeoff by comparing a low-fidelity model correction with a single-fidelity quadratic approximation correction that is often employed.
In this figure, the low-fidelity model is corrected at the red point in (b)-(d). In Figure \ref{fig:example_system}(b), when the function is corrected far from the high-fidelity minimum, the minimum of the corrected function is much closer to the high-fidelity minimum than the quadratic approximation.
When corrected near the solution, however, the quadratic approximation can outperform the corrected low-fidelity function, sometimes significantly, as seen in Figure \ref{fig:example_system}(c) and the close-up in Figure \ref{fig:example_system}(d).
This example demonstrates how the utility of a low-fidelity trust-region approach drops near the local minimum, but can be beneficial when starting far away. In other words, single-fidelity approaches will often outperform multi-fidelity approaches near the local minimum.

In the next section, we introduce multi-fidelity uncertainty quantification methods that exploit low-fidelity information for estimation procedures.

\subsection{Multi-fidelity uncertainty quantification}\label{sec:UQ}
Optimization under uncertainty minimizes a risk metric \cite{roald2023power}. Risk metrics can be computationally expensive to evaluate in high-dimensional problems because these settings rely on Monte Carlo sampling. To estimate the risk metrics more accurately under a fixed budget, low-fidelity models can again be used for precise, but biased, estimates.
Multi-fidelity estimation methods have been developed to provide unbiased estimates with potentially orders-of-magnitude of variance reduction for mean, variance, and other sampling-based estimators \cite{giles2008multilevel,gorodetsky2020generalized, dixon2024covariance}.

Formally, let $\xi$ be a vector of system parameters that is represented by a random variable with sample space $\Xi$ and known probability distributions.
We broaden the definition of the expensive simulation $f(x,\xi):\mathbb{R}^d\times \Xi\to\Theta$ to include the uncertain parameters $\xi$.
Since $\xi$ is a random variable, the simulation $f(x,\xi)$ also results in a random variable with sample space $\Theta$.
The probability density of $f(x,\xi)$ cannot be explicitly evaluated, however, when $f$ is a black-box simulation.
To quantify the uncertainty in $f(x,\xi)$, we estimate a statistic $\mathbf{R}[\cdot]$ that summarizes the random variable distribution, such as an expectation $\mathbf{R}[\cdot]=\mathbb{E}[\cdot]$ or variance $\mathbf{R}[\cdot]=\mathbb{V}ar[\cdot]$.
In the next section, we instead estimate the uncertainty on the objective $J$ to introduce optimization under uncertainty, but focus on estimating $\mathbf{R}[f(x,\xi)]$ in this section.

A common approach to estimating a risk metric on $f(x,\xi)$ is with sampling-based estimators, such as Monte Carlo.
Sampling approaches create $N$ independent and identically distributed (i.i.d.) random variables of the system parameter $\{\xi_1,\xi_2,\cdots,\xi_N\}$ and pass them to the simulation $\{f(x,\xi_1),f(x,\xi_2),\cdots,f(x,\xi_N)\}$.
Given the many model random variables, the statistic can be estimated with an estimator $\mathbf{\hat{R}}$.
For example, the mean estimator becomes $\mathbf{\hat{R}}[f(x,\xi)] = \frac{1}{N}\sum_{i=1}^N f(x,\xi_i).$
Since the estimate $\mathbf{\hat{R}}[f(x,\xi)]$ is a function of $N$ random variables, the estimator itself is a random variable with sample space $\Theta$.
To obtain an evaluation of the statistic, we can sample from $\mathbf{\hat{R}}[f(x,\xi)]$ by sampling the $N$ i.i.d. parameter variables $\xi$ from their distributions, passing them through the expensive simulation $f$, and calculating the estimate for those $N$ samples.

Multi-fidelity methods provide improved estimators that reduce the number of high-fidelity simulations \cite{giles2008multilevel,peherstorfer2016optimal,gorodetsky2020generalized}. Most multi-fidelity estimators~\cite{giles2008multilevel,peherstorfer2016optimal,gorodetsky2020generalized,gorodetsky2024grouped} are built on the concept of control variates ~\cite{nelson1990control}. Furthermore, they can be employed for a wide range of statistics that can include mean, variance, quantiles, sensitivities indices, probabilities of failure, and more~\cite{gorodetsky2020generalized, qian2018multifidelity, ayoul2023quantifying, peherstorfer2016multifidelity}.

As a simple example by which to demonstrate the approach in this paper, we use multilevel Monte Carlo (MLMC). MLMC was introduced by Giles \cite{giles2008multilevel}, and the method linearly combines high- and low-fidelity models to create a multi-fidelity estimator.
With two models, the MLMC estimator becomes
\begin{align}
    \mathbf{\hat{R}}_{\textrm{ML}}(x) = \mathbf{\hat{R}}^N[f(x,\xi)] - (\mathbf{\hat{R}}^N[f_L(x,\xi)] - \mathbf{\hat{R}}^M[f_L(x,\xi)]) \label{eq:MLMC_func}
\end{align}
where $\mathbf{\hat{R}}^N[f_L(x,\xi)]$ represents a low-fidelity risk estimator that uses $N$ samples. 
The last term in the MLMC estimator is the low-fidelity estimator that uses $M>N$ samples, which are independent from the original $N$ samples.
When the low-fidelity model $f_L$ has similar variance and high correlation to the high-fidelity model $f$, an MLMC mean estimator $\mathbf{\hat{R}}_{\textrm{ML}}$ is able to provide a more accurate estimator than an equivalent-cost high-fidelity estimator. In this work, we rely on the sub-optimal MLMC estimator for variance reduction, because more general methods cited above often require known correlations between models. We leave the complexities associated with estimating these correlations to future work.

In the next section, we introduce single-fidelity optimization under uncertainty.

\subsection{Optimization under uncertainty}\label{sec:OoU}

~~~Optimization under uncertainty (OuU) minimizes a risk metric of a random objective function.  In this setting, the objective $J$ is a function of the random simulation $f(x,\xi)$. Therefore, the objective function $J(x,\xi,f(x,\xi)):\mathbb{R}^d \times\Xi\times\Theta\to\Omega$ is also a random variable with sample space $\Omega$. Since minimizing a random variable is ill-defined, OuU minimizes a statistic of the random objective instead, and this statistic is often called a {\it risk metric}. Formally, we consider unconstrained minimization of a risk metric $\mathbf{R}:\Omega\to\mathbb{R}$ that maps the random objective to a scalar
\begin{align}
    \min_x ~\mathbf{R}[J(x,\xi,f(x,\xi))]. \label{eq:true_objective}
\end{align}
As an example, the expected value $\mathbb{E}[J(x,\xi,f(x,\xi))]$ is used for risk-neutral optimization. However many other choices such as a linear combination of expectation and variance, conditional value at risk, or probability-based statistics are possible \cite{roald2023power,menhorn2024multilevel}. 
In this work, we solely consider unconstrained optimization, but random constraints have been investigated with summary statistics or probability-based constraints \cite{tong2022optimization}.

When the risk metric $\mathbf{R}(x)$ can be explicitly evaluated as a function of the decision variables $x$ (without any sampling procedures), Equation \eqref{eq:true_objective} is a deterministic optimization problem that can be solved by many off-the-shelf nonlinear programming procedures. The risk metric, however, cannot often be evaluated explicitly as a function of $x$.  The most common approach to approximating the risk metric is with sampling procedures, such as Monte Carlo, which are inherently noisy. For example, in the risk-neutral case where $\mathbf{R}[\cdot] = \mathbb{E}[\cdot]$, we can estimate the mean with $N$ random samples $\{\xi_1,\cdots,\xi_N\}$ of the uncertain parameter
\begin{align}
    \mathbf{\hat{R}}[J(x,\xi,f(x,\xi))] = \frac{1}{N}\sum_{i=1}^N J(x,\xi_i,f(x,\xi_i))
\end{align}
where $\mathbf{\hat{R}}$ approximates $\mathbf{R}$.
Variances, probabilities of failures, and quantiles can all be estimated using sampling-based estimators for risk-aware and robust optimization objectives \cite{menhorn2024multilevel,chaudhuri2020risk,qiu2015robust}. 
However, since the samples $\{\xi_1,\cdots,\xi_N\}$ are randomly drawn, $\mathbf{\hat{R}}$ is noisy. 
Thus, the estimate $\mathbf{\hat{R}}$ is a random variable, dependent on the samples drawn for estimation.

Optimization over a random objective $\mathbf{\hat{R}}$ is considered under the class of stochastic optimization, which handles any stochastic behavior during the optimization procedure. To generalize, the two most common approaches to solve stochastic optimization under uncertainty problems may be grouped into SA methods \cite{bottou2018optimization,fehrman2020convergence,chen2018stochastic,pasupathy2018sampling,curtis2019stochastic,blanchet2019convergence,byrd2012sample,wang2013variance,kouri2013trust}
and scenario-based optimization (SAA) \cite{fu2015handbook,homem2014stochastic,shapiro1991asymptotic,homem2014stochastic, ermoliev2013sample}. 
SA methods generally take independent samples at every optimization iteration and guarantee convergence to the minimum of the true risk metric \cite{bottou2018optimization,fehrman2020convergence}.
By drawing new samples of $\xi$ at every iteration, SA methods converge to the true solution $x^*=\min \mathbf{R}(x)$, even though estimates $\mathbf{\hat{R}}$ are used at every iteration.
The most common SA procedure is stochastic gradient descent (SGD), which generates an estimate of the gradient of the risk metric with an independent set of samples at each iteration.
Common SA approaches also increase the number of samples per iteration to ensure convergence \cite{ha2025multi, ha2024adaptive, van2021mg, kouri2014multilevel, kouri2013trust, chen2018stochastic, deng2009variable}.

Scenario-based optimization, on the other hand, fixes the sample set of the random variable $\xi$ before optimization, transforming the stochastic problem into a deterministic one. 
This introduces bias, but the solution converges to the true optimum as the sample size increases \cite{shapiro1991asymptotic,homem2014stochastic}.
By fixing the set of samples before optimization, scenario-based optimization converges to the solution $x^*(\xi)=\min \mathbf{\hat{R}}(x;\xi)$ which is a function of the set of random variables drawn. 
Thus, the optimal solution is a random variable whose variance can be decreased as the size of the sample set increases. Removing the stochasticity from the problem enables any off-the-shelf nonlinear programming implementation to work efficiently on optimization under uncertainty. 

In this work, we introduce a hybrid approach that uses SA to draw new samples at every iteration, but relies on SAA to propose steps using a trust-region.
In the next section, we introduce previous works on the union of OuU with multi-fidelity methods.

\section{Related Work}\label{sec:MFOuU}
~Multi-fidelity optimization under uncertainty (MFOuU) methods reduce the cost of optimization by exploiting the relationship between the high- and low-fidelity models. 
There are two existing strategies to exploit a low-fidelity model $f_L(x,\xi)$ in MFOuU methods.
These strategies directly correspond to the two function inputs ($x$ and $\xi$).
The first strategy is to exploit the ``correlation" between $f(x,\cdot)$ and $f_L(x,\cdot)$ with respect to the random variable $\xi$.
This is a measure of how related the functions are across the random space.
The second strategy is to exploit the ``curvature" between $f(\cdot,\xi)$ and $f_L(\cdot,\xi)$ with respect to the decision variable $x$.
This is a measure of how related the functions are across the design space. 
Most existing works can be described by how they exploit the correlation or curvature between fidelities.
Previous works for multi-fidelity OuU can be seen in Table \ref{tab:MFOuU}.
This table summarizes each method by their use of trust-regions, gradients, and their exploitation of correlation and curvature.

\begin{figure*}[h]
\begin{center}
\captionof{table}{
Previous works on multi-fidelity optimization under uncertainty. Methods are compared by their use of trust-regions and gradients, and their exploitation of correlation and curvature.
}
\tempdima=  \dimexpr \textwidth/7 - 2\tabcolsep\relax
\begin{tabular}{|M{3.5cm}||M{1.75cm}|M{1.75cm}||M{1.75cm}|M{1.75cm}|} 
 \hline
 {\textbf{Method}} &
 {\textbf{Trust-region}} & 
 {\textbf{Gradient-based}} &  
 {\textbf{Exploits Correlation}} & 
 {\textbf{Exploits Curvature}} \\ 
 \hline \hline
 Dereich et al. \cite{dereich2019general} (MLMC-SGD) &  & \cmark & \cmark &  \\ \hline
 Frikha \cite{frikha2016multi} &  & \cmark & \cmark &  \\ \hline
 Korondi et al. \cite{korondi2021multi} &   &  &  & \cmark\\ \hline
 Ng and Willcox \cite{ng2014multifidelity} &  &  & \cmark & \\\hline
 Menhorn et al. \cite{menhorn2024multilevel} & \cmark &  &  \cmark & \\\hline
 Agrawal et al. \cite{agrawal2023multi} &  &  &  \cmark &  \\\hline
 De et al. \cite{de2020bi} (BF-SVRG) &  & \cmark &  \cmark &  \\ \hline
 Shah et al. \cite{shah2015multi} &  &  &  \cmark & \cmark \\\hline
 Van Barel et al. \cite{van2021mg} &  & \cmark &  \cmark & \cmark \\\hline
 Ha and Mueller \cite{ha2024adaptive} (ASTRO-BFDF) & \cmark &  &  \cmark & \cmark \\\hline
 Ha and Mueller \cite{ha2025multi} & \cmark &  &  \cmark & \cmark \\\hline
 \textbf{BISTRO} & \cmark & \cmark &  \cmark & \cmark\\
 \hline
\end{tabular}
\label{tab:MFOuU}
\end{center}
\end{figure*}

To exploit curvature in the design space, most methods in Table \ref{tab:MFOuU} first convert the OuU into a deterministic optimization problem using methods like polynomial chaos expansion to calculate risk metrics \cite{korondi2021multi,shah2015multi}.
Once the stochasticity is removed from the problem, multi-fidelity optimization methods, such as those introduced in Section \ref{sec:MFO}, are applied to exploit curvature.
Examples of exploiting curvature include trust-region \cite{ha2025multi} or Bayesian optimization \cite{korondi2021multi} methods that share information in the design space. 

To exploit correlation in the random variable space, most methods in Table \ref{tab:MFOuU} use existing OuU methods (Section \ref{sec:OoU}) with a multi-fidelity sampling-based estimator, typically with MLMC \cite{dereich2019general, frikha2016multi,menhorn2024multilevel,agrawal2023multi} or approximate control variates (ACV) \cite{ng2014multifidelity,de2020bi,ha2025multi}.
Section \ref{sec:UQ} discussed the multi-fidelity estimation methods that are common to exploit the correlation between fidelities, which are then used in methods such as SGD.

To the best of the authors' knowledge, only a few existing methods \cite{shah2015multi,ha2025multi, ha2024adaptive, van2021mg, kouri2014multilevel} exploit both correlation and curvature simultaneously. 
Shah et al. \cite{shah2015multi} develop a global surrogate of the objective over the joint design and random variable space using a fixed set of high-fidelity evaluations. 
They then use a scenario-based optimizer to find a solution. 
Because their surrogate is trained with finite data, their final result remains biased. 
Ha and Mueller \cite{ha2025multi, ha2024adaptive} consider an alternate approach that is more similar to ours. 
They use a stochastic trust-region process with local low-fidelity evaluations to inform the high-fidelity search within a novel dual trust-region strategy. 
However, their approach is a derivative-free method, requiring a batch of design points to approximate the models at every optimization iteration.
Finally, MLMC extensions of the MG/OPT \cite{lewis2000multigrid} optimization algorithm use MLMC to exploit correlation while using MG/OPT to exploit curvature \cite{van2021mg, kouri2014multilevel}.

BISTRO deviates from these previous SA approaches \cite{ha2025multi, ha2024adaptive,van2021mg, kouri2014multilevel} since it does not increase the number of samples drawn per iteration.
Prior works adaptively grow the sample size to maintain high estimator accuracy and satisfy a sufficient decrease condition. 
However, SA methods with variable sample sizes may yield suboptimal convergence rates when measured in terms of total simulation evaluations. 
Achieving optimal efficiency in these adaptive schemes requires carefully tuned, problem-dependent sampling rates \cite{pasupathy2018sampling}.
BISTRO avoids finding the optimal sampling rate by employing SGD, which only uses one high-fidelity simulation per iteration. 
While SGD achieves optimal efficiency given an appropriate, problem-dependent learning rate and decay schedule \cite{bottou2018optimization}, this shifts the challenge from tuning sampling rates to tuning the learning rate schedule.
As with adaptive sampling in previous work, a variety of methods exist to adaptively adjust the learning rate in SGD-based approaches \cite{kingma2014adam, loizou2021stochastic, ge2019step} which will be investigated in BISTRO in future work. 

In the numerical results in Section \ref{sec:results}, we compare against multiple existing methods. Dereich \cite{dereich2019general} uses an MLMC estimator with an SGD optimizer to estimate the high-fidelity gradient (MLMC-SGD), exploiting the correlation between fidelities. We compare against a simplification of \cite{dereich2019general} since we do not adaptively choose the number of fidelities and number of samples per fidelity. De et al. \cite{de2020bi} propose BF-SVRG, an extension to \cite{dereich2019general} that combines ACV estimation with a variance reduction strategy, SVRG. Lastly, we compare against ASTRO-BFDF by Ha and Mueller \cite{ha2024adaptive} which is the derivative-free trust-region approach described above.
We do not compare against Shah et al. \cite{shah2015multi} since it is a SAA approach that does not converge to the underlying optimum or against the MG/OPT methods \cite{van2021mg,kouri2014multilevel} since they are only applicable to PDE-constrained optimization with coarser PDE discretizations as low-fidelity models.

In this work, we propose combining a bi-fidelity SGD method \cite{dereich2019general} with a fixed-trust-region approach to exploit both the curvature and correlation to improve the cost of stochastic optimization.
The results in Section \ref{sec:results} demonstrate that our trust-region procedure allows BISTRO to outperform these existing methods.

\section{BISTRO - Algorithm}\label{sec:method}

We propose a bi-fidelity stochastic approximation method that leverages trust-region optimization. The basic idea of BISTRO is to warm-start a variance-reduced SGD optimizer close to the solution using an initial trust-region phase. The initial trust-region optimizer uses the low-fidelity model to exploit the curvature in $x$.
Once near the solution, the trust-region optimizer terminates and an SGD optimizer with an MLMC gradient estimator begins.
We assume the curvature of the low-fidelity model is not informative near the high-fidelity solution --- assuming otherwise would imply that only the low-fidelity model needs to be used.
As a result, we stop using the trust-region optimizer near the solution because the trust-region suboptimization can be quite expensive if the cost of the low-fidelity model is non-negligible.

The overall algorithm for BISTRO is provided in Algorithm \ref{alg:BISTRO}. 
It begins with using a trust-region optimizer, which is discussed in Section \ref{sec:trust-region}.
The variance-reduced SGD procedure and the criteria for switching is detailed in Section \ref{sec:mlmc-sgd}.

\begin{algorithm}[h]
\caption{BISTRO}\label{alg:BISTRO}
\begin{algorithmic}[1]
\REQUIRE High-fidelity random objective $J_H$, Low-fidelity random objective $J_L$, Initial design $x_0$, Trust-region size $\Delta$, Computational budget $B$, Number of high-fidelity samples per iteration $N$, Number of low-fidelity samples per iteration $M$, Number of low-fidelity samples per trust-region solve $K$, SGD learning rate parameter $\beta$, SGD learning rate parameter $\gamma$, trust-region learning rate $\lambda$

\STATE Begin by using the trust-region optimizer.

\noindent
$\texttt{use\_trust\_region} \gets \texttt{True}$, $k\gets0$, $t\gets0$ 

\STATE Draw new estimators $\mathbf{\hat{R}}_{H}^N,\nabla\mathbf{\hat{R}}_{H}^N$ and $\mathbf{\hat{R}}_L^K, \nabla\mathbf{\hat{R}}_L^K,\mathbf{\hat{R}}_L^M, \nabla\mathbf{\hat{R}}_L^M$ at $x_0$.
\STATE Create the MLMC estimator $\mathbf{\hat{R}}_{ML}^{(0)}(x_0)$ from Equation \eqref{eq:grad_mlmc}.
\WHILE{budget $B$ remains}

\IF{\texttt{use\_trust\_region}}

\STATE Create the corrected low-fidelity function $\mathbf{\hat{R}}_{C}^{(k)}(x_k)$ from Equation \eqref{eq:corr_surr}.
\STATE Find $s_{k}$ by solving the trust-region sub-problem from Equation \eqref{eq:trust_region_optimization}. \label{line:trust_region_proposal}
\begin{align}
    s_{k}  \gets \arg\min_{s_{k}} ~\mathbf{\hat{R}}^{(k)}_C(x_{k}+s_k) \quad\quad\textrm{s.t.}\quad\quad||s_k||\leq\Delta 
\end{align}
\STATE Take the new step $x_{k+1} = x_k + \lambda s_k$
\STATE Draw new estimators $\mathbf{\hat{R}}_{H}^N,\nabla\mathbf{\hat{R}}_{H}^N$ and $\mathbf{\hat{R}}_L^K, \nabla\mathbf{\hat{R}}_L^K,\mathbf{\hat{R}}_L^M, \nabla\mathbf{\hat{R}}_L^M$ at $x_{k+1}$.
\STATE Create the MLMC estimator $\mathbf{\hat{R}}^{(k+1)}_{ML}(x_{k+1})$ from Equation \eqref{eq:grad_mlmc}.

\IF{$\mathbf{\hat{R}}_{ML}^{(k+1)}(x_{k+1})>\mathbf{\hat{R}}^{(k)}_{ML}(x_k)$}\label{line:switch}
\STATE $\texttt{use\_trust\_region} \gets \texttt{False}$
\ENDIF
\ELSE
\STATE Draw new estimators $\nabla\mathbf{\hat{R}}_{H}^N$ and $\nabla\mathbf{\hat{R}}_L^N,\nabla\mathbf{\hat{R}}_L^M$ at $x_k$.
\STATE Create the MLMC estimator $\nabla\mathbf{\hat{R}}^{(k)}_{ML}(x)$ from Equation \eqref{eq:grad_mlmc}.
\STATE Update the learning rate to $\alpha_t \gets \beta/(\gamma + t)$
\STATE Find the MLMC-SGD step according to $x_{k+1} \gets x_k-\alpha_t \nabla\mathbf{\hat{R}}^{(k)}_{ML}(x_k)$.
\STATE Update the learning rate counter $t\gets t+1$
\ENDIF
\STATE $k\gets k+1$
\ENDWHILE
\RETURN Final design $x_k$
\end{algorithmic}
\end{algorithm}

\subsection{Trust-Region Phase}\label{sec:trust-region}

~~BISTRO initially uses a corrected low-fidelity model within a trust-region of fixed size
to propose a new step, similar to the approach described in Section \ref{sec:MFO}.
To merge notation from the previous sections, let $J_H(x,\xi)\equiv J(x,\xi,f(x,\xi))$ and $J_L(x,\xi)\equiv J(x,\xi,f_L(x,\xi))$.
Specifically, 
the trust-region optimizer minimizes $\mathbf{R}_H(x) \equiv \mathbf{R}[J_H(x,\xi)]$,
through a sequence of sub-problems that minimize a corrected low-fidelity risk metric. 
To satisfy the requirements of Equation \eqref{eq:trust_requirements}, 
the corrected function must match the value and derivative 
of $\mathbf{R}_H$ at the center $x_0$ of a trust-region.
The corrected risk metric is
\begin{align}
    \mathbf{R}_{C}(x) &\equiv \mathbf{R}_L(x) + (\mathbf{R}_H(x_0)\!-\!\mathbf{R}_L(x_0)) 
    + (\nabla\mathbf{R}_H(x_0)\!-\!\nabla\mathbf{R}_L(x_0) )^\top(x\!-\!x_0) , \label{eq:true_corrected_function}
\end{align}
where $\mathbf{R}_L  \equiv\mathbf{R}[J_L(x,\xi)]$.
The term $(\mathbf{R}_H(x_0)\!-\!\mathbf{R}_L(x_0))$ corrects the low-fidelity value at $x_0$ while the last term corrects its gradient.
To evaluate the corrected function $\mathbf{R}_{C}$ at a new location $x'$, only the low-fidelity function needs to be evaluated, since the values and gradients of $\mathbf{R}_H$ and $\mathbf{R}_L$ have already been stored at the center $x_0$.
Equivalently, the corrected function can be rewritten as
\begin{align}
    \mathbf{R}_{C}(x) &= \underbrace{\mathbf{R}_H(x_0) + \nabla\mathbf{R}_H(x_0)(x-x_0)}_{\text{Linearized High-fidelity}}
    + \underbrace{\mathbf{R}_L(x) -\!\mathbf{R}_L(x_0) 
    -\!\nabla\mathbf{R}_L(x_0)^\top(x\!-\!x_0)}_{\text{Higher-order Moments of Low-fidelity}} ,\nonumber
\end{align}
as the linearized high-fidelity function around $x_0$ with the higher-order moments from the low-fidelity function.
Thus, $\mathbf{R}_{C}$ approximates $\mathbf{R}_{H}$ near $x_0$ and is able to exploit the curvature information from the low-fidelity.

Since $\mathbf{R}_H$ and $\mathbf{R}_L$ cannot be evaluated exactly in practice, sampling-based estimators are used. Let $\mathbf{\hat{R}}_H^N$ be an $N$-sample estimator of $\mathbf{R}_{H}$ and $\nabla\mathbf{\hat{R}}_{H}^N$ be its gradient.
BISTRO assumes that the gradient of the risk metric estimator can be evaluated.
This, in turn, requires the gradient of $J_H$ and $J_L$ to also be evaluated.
For example, if the risk is an expectation $\mathbf{R}[\cdot]=\mathbb{E}[\cdot]$, then
\begin{align}
    \nabla_x\mathbf{\hat{R}}_{H} = \nabla_x \left(\frac{1}{N}\sum_{i=1}^N J_H(x,\xi)\right) =  \frac{1}{N}\sum_{i=1}^N \nabla_x J_H(x,\xi)
\end{align}
 Thus, $\nabla_x J_H(x,\xi)$ may be found using automatic or numerical differentiation. 
The gradients of more complex risk metrics must first be found before gradient-based stochastic optimization can be applied for those metrics.

We now estimate the corrected low-fidelity function in Equation \eqref{eq:true_corrected_function} using the MC estimators.
For $K,N\in\mathbb{Z}_+$, the corrected function becomes
\begin{align}
    \mathbf{\hat{R}}_{C}^{(k)}(x) &= \mathbf{\hat{R}}_L^K(x)     
    + (\mathbf{\hat{R}}_{H}^N(x_0) - \mathbf{\hat{R}}_L^K(x_0) )
    + (\nabla\mathbf{\hat{R}}_{H}^N(x_0) - \nabla\mathbf{\hat{R}}_L^K(x_0) )^\top(x-x_0) \label{eq:corr_surr}
\end{align}
where the superscript $(k)$ denotes the iteration index of Algorithm \ref{alg:BISTRO} and each $(k)$ is evaluated with a new set of samples. 
We minimize \eqref{eq:corr_surr} over a trust-region, centered at $x_0$, to propose a new step for the optimization procedure.
For the proof of convergence in Section \ref{sec:theo_results}, the set of $N$ and $K$ samples must be independent. 
In practice, however, if the $K$ samples used to evaluate the low-fidelity on $x_0$ are equal to the $N$ high-fidelity samples on $x_0$, correlation can be exploited for variance reduction.
To find the minimum of \eqref{eq:corr_surr}, we use scenario-based optimization and a deterministic black-box optimizer to find a proposal step. 
In summary, the trust-region component in BISTRO finds a step $s_k$
\begin{align}
    s_{k} =\arg\min_{s_k} ~\mathbf{\hat{R}}^{(k)}_C(x_k+s_k) \label{eq:trust_region_optimization}\quad\quad\textrm{s.t.}\quad\quad||s_k||\leq\Delta
\end{align}
for a trust-region size of $\Delta\in\mathbb{R}_+$ centered around the current location $x_k$ for the $k$-th iteration. 
The solution to the minimization $s_k$ is the direction for the $k+1$th step $x_{k+1} = x_k + \lambda s_k$ where $\lambda\in\mathbb{R}_+$ is a learning rate.
This learning rate is necessary to provide convergence guarantees in Section \ref{sec:trust-conv}, but is set to $\lambda=1$ practically when oracle constants are unknown.
Further, the trust-region size is constant throughout the procedure, unlike the adaptive sizing in previous works, since BISTRO does not require the trust-region size to decay for convergence.
The trust-region is included to ensure that assumptions about the low-fidelity model are met (see Section \ref{sec:trust-conv}) and arbitrarily large step sizes are not taken.

\subsection{MLMC-SGD Phase}\label{sec:mlmc-sgd}

After the trust-region phase terminates, BISTRO switches to a multi-fidelity variance-reduced SGD algorithm to exploit the correlation between the fidelities.
Instead of taking an SGD step with the high-fidelity gradient $\nabla\mathbf{\hat{R}}^N_H(x)$, we use the MLMC estimator $\nabla\mathbf{\hat{R}}^{(k)}_{ML}$, which provides a low-variance estimate of the gradient.
This approach is identical to the approach from \cite{dereich2019general} without the adaptive fidelity selection and adaptive sampling rates. While BISTRO may benefit from these features, the analysis and implementation will be left to future work.
To exploit the correlation between $J_H(x,\xi)$ and $J_L(x,\xi)$ with respect to $\xi$, MLMC is used to reduce the variance of the gradient of the estimator $\nabla\mathbf{\hat{R}}_H$ and the risk estimator $\mathbf{\hat{R}}_H$.
Thus, the MLMC estimators are
\begin{align}
    \mathbf{\hat{R}}^{(k)}_{ML}(x) &= \mathbf{\hat{R}}_H^N(x) - (\mathbf{\hat{R}}_L^N(x) - \mathbf{\hat{R}}_L^M(x)) \\
    \nabla\mathbf{\hat{R}}^{(k)}_{ML}(x) &= \nabla\mathbf{\hat{R}}_H^N(x) - (\nabla\mathbf{\hat{R}}_L^N(x) - \nabla\mathbf{\hat{R}}_L^M(x)) \label{eq:grad_mlmc}
\end{align}
for Monte Carlo sample sizes satisfying $M>N$.
The resulting MLMC estimators $\nabla\mathbf{\hat{R}}^{(k)}_{ML}(x)$ and $\mathbf{\hat{R}}^{(k)}_{ML}(x)$ are unbiased with respect to the high-fidelity estimators $\nabla\mathbf{\hat{R}}^N_{H}(x)$ and $\mathbf{\hat{R}}^N_{H}(x)$ under mild regularity assumptions\footnote{
Assuming for all $x\in\mathbb{R}^d$ and  almost all $\xi\in\Xi$: the function $J_H(x,\xi)$ is Lebesgue-integrable with respect to $\xi$, the gradient $\nabla_x J_H(x,\xi)$ exists, and there exists an integrable function $g(\xi):\Xi\to\mathbb{R}^d$ such that $|\nabla_x J_H(x,\xi)|\leq g(\xi)$ \cite{asmussen2007stochastic}.
}.
If the low-fidelity function were to increasingly improve its approximation of the high-fidelity function, the variance of the MLMC estimator disappears.
Thus, the reduced-variance SGD method approaches traditional gradient descent, without any noise. 
Further, we assume that the $N$ samples in the MLMC estimator are the same set of $N$ samples used in trust-region corrected function, but this is not necessary for convergence.
We denote the use of an MLMC estimator with SGD as MLMC-SGD. 

As seen in Algorithm \ref{alg:BISTRO}, the learning rate in MLMC-SGD begins to decay as a function of the number of iterations after the trust-region terminates.
This is common in most SGD approaches and allows MLMC-SGD to converge only when $\alpha_t$ converges to zero. 
The optimal rate at which $\alpha_t$ converges depends on the differentiability, convexity, and noise of the risk metric estimator \cite{bottou2018optimization}.
In practice, the learning rate can be chosen through adaptive schemes or tuned manually.
Finally, BISTRO terminates when the computational budget of the algorithm is reached, but other convergence criteria from the SGD literature may be used.

\subsection{Switching Condition}

The trust-region phase is eventually terminated due to its cost and deteriorating performance.
The trust-region optimizer in Line \ref{line:trust_region_proposal} in Algorithm \ref{alg:BISTRO} is more expensive than an MLMC-SGD step since the corrected function $\mathbf{\hat{R}}^{(k)}_C$ requires many evaluations of $\mathbf{\hat{R}}^K_L$.
The trust-region optimizer may also propose hindering directions if there is no useful curvature information that the low-fidelity model can provide.
Specifically, near the high-fidelity solution, the approximation $\mathbf{\hat{R}}^{(k)}_C$ is not expected to provide useful curvature information.

To determine when to terminate the trust-region optimizer, the \textit{switching condition} is introduced. 
This condition estimates the proximity to the high-fidelity solution and is motivated in Section \ref{sec:motivation}.
If the switching condition is met on Line \ref{line:switch} in Algorithm \ref{alg:BISTRO}, BISTRO asymptotically relies solely on MLMC-SGD to converge. 
Let the superscript $(k)$ denote independent samples from $(k+1)$. When $\mathbf{\hat{R}}^{(k+1)}_{ML}(x_{k+1}) \geq \mathbf{\hat{R}}^{(k)}_{ML}(x_{k})$, BISTRO stops using the trust-region optimizer and begins using MLMC-SGD only.
This condition could be met due to the noise of the estimator $\mathbb{V}ar[\mathbf{\hat{R}}^{(k)}_{ML}]$ being non-negligible or if the trust-region proposes a poor direction $\mathbf{{R}}_H(x_{k+1}) \geq \mathbf{{R}}_H(x_{k})$.
In either case, the trust-region should be terminated so the MLMC-SGD can exploit the correlation for variance reduction.
See Section \ref{sec:motivation} for a theoretical motivation of this condition, and Section \ref{sec:quadratic_problem} for its empirical analysis. 

\section{Theoretical Results}\label{sec:theo_results}
We provide guarantees that BISTRO in Algorithm \ref{alg:BISTRO} converges to a stationary point in expectation under regularity assumptions, such as continuous differentiability and bounded estimator noise. 
Since BISTRO is a two-phase algorithm (the trust-region phase and the MLMC-SGD phase), the convergence proof relies on two main arguments.
First, we provide the trust-phase convergence guarantee under regularity assumptions to a basin near the stationary point.
We then show that MLMC-SGD asymptotically converges to the stationary point.
By guaranteeing that the trust-phase will terminate and MLMC-SGD will begin, we guarantee BISTRO convergence to a stationary point.

We begin by demonstrating the trust-phase convergence around a stationary point in Section \ref{sec:trust-conv}.
We then provide details of the asymptotic MLMC-SGD convergence to a stationary point in Section \ref{sec:SGD_MLMC}.
Finally, we provide motivation for the switching condition in Section \ref{sec:motivation}.

\subsection{Trust-Region Convergence}\label{sec:trust-conv}
We provide convergence guarantees in expectation to an area near a stationary point for the trust-region phase.
Our strategy is to show that the trust-region phase is similar to a stochastic Newton's method and prove convergence in a similar manner.
Convergence of any stochastic descent algorithm has been proven with Theorem 4.8 in Bottou et al. \cite{bottou2018optimization} under regularity assumptions.
We follow a similar outline and make the following assumptions about continuous differentiability, low-fidelity convexity, and estimator regularity.

\begin{defin}
    \label{defin:convex}
    A function $f:\mathbb{R}^d\to\mathbb{R}$ is strongly convex if there exists a constant $c>0$ such that for all $x,x'\in\mathbb{R^d}$
    \begin{align}
        f(x) \geq f(x') + \nabla f(x')^T(x-x') + \frac{c}{2}|| x-x'||^2_2.
    \end{align}
\end{defin}

\begin{defin}
    \label{defin:convex}
    A function $f:\mathbb{R}^d\to\mathbb{R}$ has a Lipschitz continuous gradient $\nabla f$ if there exists a constant $L>0$ such that for all $x,x'\in\mathbb{R^d}$
    \begin{align}
        ||\nabla f(x) - \nabla f(x')||_2 \leq L||x-x'||_2.
    \end{align}
\end{defin}

\begin{assumption}
    \label{ass:hf_smooth_convex}
    $\nabla_x \mathbf{R}_H$ is Lipschitz continuous with constant $L_H$.
\end{assumption}

\begin{assumption}
    \label{ass:lf_smooth_convex}
    The estimator $\mathbf{\hat{R}}^{M,(k)}_L$ is strongly convex with convexity constant $c_L$ and $\nabla \mathbf{\hat{R}}^{M,(k)}_L$ is Lipschitz continuous with constant $L_L$ for all $k$.
\end{assumption}

\begin{assumption}
    $\mathbf{R}_H(x)$ is bounded below for all $x\in\mathbb{R}^d$ and its infimum is $\mathbf{R}_\star=\inf_x \mathbf{R}_H(x)$.
    \label{ass:obj_bound}
\end{assumption}

\begin{assumption}
    \label{ass:var_bound}
    Let $W\geq 0$ and $W_V\geq 0$ such that for a gradient of an estimator $\nabla\mathbf{\hat{R}}^N_H$ and all $x\in\mathbb{R}^d$
    \begin{align}
     Tr(\mathbb{V}ar[\nabla\mathbf{\hat{R}}^N_H(x)])= \mathbb{E}[||\nabla\mathbf{\hat{R}}^N_H(x)||_2^2] - ||\mathbb{E}[\nabla\mathbf{\hat{R}}^N_H(x)]||_2^2 \leq W + W_V ||\nabla\mathbf{R}_H(x)||_2^2.\nonumber
    \end{align}
\end{assumption}

Assuming continuous differentiability and the risk metric lower-bound are standard for general optimization convergence results. 
Further, assuming the noise of the gradient estimator is bounded is also common for SGD results \cite{bottou2018optimization,curtis2019stochastic}. 
A more restrictive assumption is strong convexity for every low-fidelity risk estimate. 
For scenario-based optimization, where only one realization of an estimator is used for the entire optimization procedure, this assumption is required for convergence.
While the trust-region bound $\Delta$ is assumed arbitrarily large in the convergence proofs below, they are necessary in a practical setting where the strong convexity of the low-fidelity risk estimates cannot be ensured.
By including the trust-region size constraint, the proposal step may still be informative near the trust-region center, even when the low-fidelity is non-convex.
We now provide convergence guarantees to show that the trust-phase converges near a stationary point where the asymptotic average gradient squared norm is bounded.

\begin{theorem}[Trust-Phase Converges]
    Using Assumptions \ref{ass:hf_smooth_convex}-\ref{ass:var_bound}, let $x_{k+1}=x_k+\lambda \arg\min_{s_k} \mathbf{\hat{R}}^{(k)}_C(x_k+s_k)$ with learning rate $\lambda=c_L^2/(L_L L_H (W_V+1))$. The average squared gradient norm is asymptotically bounded by  
    \begin{align}
        \lim_{P\to\infty} \mathbb{E}\left[\frac{1}{P}\sum_{k=1}^P||\nabla\mathbf{R}_H(x_k)||^2\right] \leq \frac{W}{W_V+1}.
    \end{align}
\end{theorem}
\begin{proof}
    Proof in Appendix \ref{app:trust_conv}.
\end{proof}

Thus, the asymptotic proximity to the stationary point depends on the bound of the estimator noise $W$ and how fast it grows away from the stationary point $W_V$.
As the number of high-fidelity samples increases $N\to\infty$, the asymptotic norm of the gradient vanishes $W\to0$. 
Further, the learning rate $\lambda$ depends on the high- and low-fidelity continuous differentiability, as well as the low-fidelity convexity.
From Assumption \ref{ass:lf_smooth_convex}, the convexity constant $c_L$ and continuous differentiability constant $L_L$ bound all realizations of low-fidelity risk estimators. Thus, as the low-fidelity noise grows, the convexity bound
$c_L$ shrinks and the differentiability bound $L_L$ grows.
Further, as the low-fidelity noise increases, the learning rate decreases and it takes more iterations to converge to a stationary point.
Conversely, as the low-fidelity estimator noise vanishes $K\to\infty$, the convexity and continuous differentiability constants converge to the true bounds of the low-fidelity risk metric.
Thus, increasing the number of low-fidelity samples $K$ only helps to converge faster to the stationary point (by increasing the learning rate) if the bounds $c_L$ and $L_L$ are smaller and larger respectively than the bounds on the true low-fidelity risk metric.

The trust-phase can also be shown to converge to the stationary point (similar to MLMC-SGD in Section \ref{sec:SGD_MLMC}) if the learning rate decays according to a specific rate, but this is not necessary to show BISTRO convergence.
Further, we provide convergence guarantees to the global minimum under a strong convexity assumption.
We make an additional assumption about the high-fidelity risk metric that is not necessary for BISTRO convergence to a stationary point, but does provide further insight into our method.
\begin{assumption}
    \label{ass:hf_convex}
    $\mathbf{R}_H$ is strongly convex with convexity constant $c_H$.
\end{assumption}

With this strict assumption, we show that the trust-phase converges near the global minimum. 

\begin{theorem}[Trust-Phase converges with convexity]
    Using Assumptions \ref{ass:hf_smooth_convex}-\ref{ass:var_bound} and \ref{ass:hf_convex}, let $x_{k+1}=x_k+\lambda \arg\min_{s_k} \mathbf{\hat{R}}^{(k)}_C(x_k+s_k)$ with learning rate set to $\lambda=c_L^2/(L_L L_H (W_V+1))$. The expected optimality gap is asymptotically bounded by 
    \begin{align}
    \lim_{k\to\infty} \mathbb{E}[\mathbf{R}_H(x_k)-\mathbf{R}_{\star}] \leq \frac{W}{2W_Gc}
\end{align}
\end{theorem}
\begin{proof}
    Proof in Appendix \ref{app:trust_convex}.
\end{proof}

Again, if we decay the learning rate $\lambda$ to a specified rate, we are able to converge exactly to the global minimum. 
For BISTRO, we only need to converge near the basin of the solution (or stationary point) and let MLMC-SGD take over asymptotically.

We now show that the switching condition $\mathbf{\hat{R}}_{ML}^{(k+1)}(x_{k+1})>\mathbf{\hat{R}}_{ML}^{(k)}(x_{k})$ 
on Line \ref{line:switch} is met with probability one. 
Assume that near the stationary point, there exists some nonzero probability $p$ that $\mathbf{\hat{R}}_{ML}^{(k+1)}(x_{k+1})>\mathbf{\hat{R}}_{ML}^{(k)}(x_{k})$.
This is a mild assumption that practically states that the high-fidelity objective function is non-degenerate. 
Thus, the probability that the switching condition is not met in $P$ steps near the stationary point is $(1-p)^P$. Thus, since $p>0$, as $P\to\infty$, the probability of the switching condition being met asymptotically is one. 

The next section provides asymptotic convergence guarantees to a stationary point for MLMC-SGD.

\subsection{MLMC-SGD Convergence}\label{sec:SGD_MLMC}
~~We provide convergence guarantees and rates of convergence for MLMC-SGD to a stationary point and convergence rates to the global mininum under strong convexity. 
Previous works have provided convergence guarantees for MLMC-SGD \cite{dereich2019general}, but we provide guarantees from a stochastic descent perspective with a direct comparison to single-fidelity SGD.
Further, we motivate the use of MLMC-SGD by providing auxiliary conditions that guarantee MLMC-SGD converges faster than SGD asymptotically.

We now make a similar assumption to Assumption \ref{ass:var_bound} that bounds the noise of the MLMC gradient estimator.

\begin{assumption}
    \label{ass:var_bound_ml}
    Let $W_{ML}\geq 0$ and $W_{V,ML}\geq 0$ such that for all $x\in\mathbb{R}^d$ and all iterations $k$, a gradient of an estimator $\nabla\mathbf{\hat{R}}^{(k)}_{ML}$ is bounded by  
    \begin{align}
    \mathbb{E}[||\nabla\mathbf{\hat{R}}^{(k)}_{ML}(x)||_2^2] - ||\mathbb{E}[\nabla\mathbf{\hat{R}}^{(k)}_{ML}(x)]||_2^2 \leq W_{ML} + W_{V,ML} ||\nabla\mathbf{R}_H(x)||_2^2.
    \end{align}
\end{assumption}

Under this assumption we obtain the following result.
\begin{theorem}[MLMC-SGD convergence]
    Using Assumptions \ref{ass:hf_smooth_convex}, \ref{ass:obj_bound}, and \ref{ass:var_bound_ml}, let the step $x_{k+1}=x_k-\alpha_k \nabla \mathbf{\hat{R}}_{ML}^{(k)}(x_k)$ with learning rate satisfying
    \begin{align}
        \sum_{k=1}^\infty \alpha_k=\infty \quad\textrm{and}\quad \sum_{k=1}^\infty \alpha_k^2<\infty
    \end{align}
    where $\alpha_k\leq1/(L_H(W_{V,ML}+1))$ for all $k$. The asymptotic expected squared gradient norm follows
    \begin{align}
        \liminf_{k\to\infty} ~\mathbb{E}[||\nabla\mathbf{R}_H(x_k)||^2] = 0.
    \end{align}
\end{theorem}
\begin{proof}
    The proof follows from an application of Theorem 4.9 in Bottou et al. \cite{bottou2018optimization} when the MLMC gradient estimator is unbiased $\mathbb{E}[\nabla \mathbf{\hat{R}}_{ML}^{(k)}(x_k)]=\nabla \mathbf{{R}}_{H}(x_k)$.
    Thus, by taking the expectation of the MLMC gradient estimator
    \begin{align}
        \mathbb{E}[\nabla\mathbf{\hat{R}}_{ML}(x)] &= \mathbb{E}[\nabla\mathbf{\hat{R}}^N_H(x)] - \mathbb{E}[\nabla\mathbf{\hat{R}}^N_L(x) - \nabla\mathbf{\hat{R}}^M_L(x)] \\
        &= \mathbb{E}[\nabla\mathbf{\hat{R}}^N_H(x)] = \nabla\mathbf{R}_H(x) \nonumber
    \end{align}
    we see that $\nabla\mathbf{\hat{R}}_{ML}(x)$ is an unbiased estimator.
\end{proof}

Therefore, when $\beta,\gamma$ in Algorithm \ref{alg:BISTRO} are sufficiently chosen to satisfy $0\leq\alpha_t\leq1/(L_H(W_{V,ML}+1))$ for all $t$, then the MLMC-SGD phase converges to a stationary point. 
Thus, using the trust-phase convergence and switching condition results in Section \ref{sec:trust-conv}, BISTRO is guaranteed to use MLMC-SGD and asymptotically converge to a stationary point.

Further, with strong convexity in Assumption \ref{ass:hf_convex}, we can show that MLMC-SGD converges with a linear rate and provides the constants of convergence. 
These constants will be used to justify the switching condition in BISTRO and help prove that MLMC-SGD converges faster than SGD asymptotically.

\begin{theorem}[MLMC-SGD with convexity converges]
    Using Assumptions \ref{ass:hf_smooth_convex}, \ref{ass:obj_bound}, \ref{ass:hf_convex}, and \ref{ass:var_bound_ml}, let the step $x_{k+1}=x_k-\alpha_k \nabla \mathbf{\hat{R}}_{ML}^{(k)}(x_k)$ with learning rate $\alpha_k=\beta/(\gamma+k)\leq 1/(L_H (W_{V,ML}+1))$ for constants $\beta>1/c_H$ and $\gamma>0$. The expected optimality gap converges with
    \begin{align}
        \mathbb{E}[\mathbf{R}_H(x_k)-\mathbf{R}_\star] \leq \frac{\nu}{\gamma + k} \label{eq:rate_of_convergence}
    \end{align}
    where the constant of convergence is
    \begin{align}
        \nu \equiv \max\left\{ \frac{\beta^2 L_H W_{ML}}{2(\beta c_H-1)}, (\gamma + 1)(\mathbf{R}_H(x_1) - \mathbf{R}_*) \right\}.
    \label{eq:conv_constant}
    \end{align}
\end{theorem}
\begin{proof}
    Application of Theorem 4.7 in Bottou et al. \cite{bottou2018optimization} when the MLMC gradient estimator is unbiased $\mathbb{E}[\nabla \mathbf{\hat{R}}_{ML}^{(k)}(x_k)]=\nabla \mathbf{{R}}_{H}(x_k)$.
\end{proof}

Now that MLMC-SGD has been shown to converge with the provided convergence rate, we can properly define the learning rate parameters $\beta$ and $\gamma$.
It is shown in Appendix \ref{app:opt_learning_rate} that the learning rate constants $\beta,\gamma$ can be optimized to minimize the optimality gap upper bound asymptotically.
The resulting optimal parameters are $\beta_\star = 2/c_H$ and $\gamma_\star = 2L_H (W_{V,ML}+1) c_H^{-1} -1$ such that the constant of convergence becomes
\begin{align}
    \nu \equiv \max\left\{ \frac{2L_H W_{ML}}{c_H^2}, \frac{2L_H (W_{V,ML}+1)}{c_H }(\mathbf{R}_H(x_1) - \mathbf{R}_*) \right\}
    \label{eq:opt_conv_const}
\end{align}
with learning rate
\begin{align}
    \alpha_k=\frac{2}{c_H(\frac{2}{c_H}L_H (W_{V,ML}+1) -1 + k)}.
    \label{eq:new_learning_rate}
\end{align}

We now motivate the use of MLMC-SGD by providing the conditions where MLMC-SGD converges faster than SGD.
Since the constant of convergence $\nu$ depends on the estimator's variance bound $W_{ML}$ (or $W$ for SGD), it can be shown that MLMC-SGD has a faster convergence than SGD under certain cost restrictions of the low-fidelity model.
\begin{theorem}[MLMC-SGD converges faster than SGD asymptotically]
\label{theo:MLMC_faster}
    Let $B\in\mathbb{R}_+$ be the budget of optimization, $C_{ML}\in\mathbb{R}_+$ be the MLMC-SGD per-iteration cost, and $C\in\mathbb{R}_+$ be the SGD per-iteration cost such that $0<C<C_{ML}<B$.
    Also, let the upper bounds for the optimality gaps of SGD and MLMC-SGD be denoted by 
    \begin{align}
        \mathbb{E}[\mathbf{R}_H(x_{k}^{SGD})-\mathbf{R}_{\star}] &\leq Q^{SGD}(B) \quad\textrm{and}\quad \mathbb{E}[\mathbf{R}_H(x_{j}^{ML})-\mathbf{R}_{\star}] &\leq Q^{ML}(B).
    \end{align}
    for SGD iteration index $k$ and MLMC-SGD iteration index $j$ for the same budget $B$.
    Under Assumptions \ref{ass:hf_smooth_convex}, \ref{ass:obj_bound}, \ref{ass:var_bound}, and \ref{ass:var_bound_ml}, if
    \begin{align}
        C_{ML}W_{ML} < CW.
        \label{eq:asy_mlmc_improve_cond}
    \end{align}
    then MLMC-SGD will have a smaller upper bound per unit cost as $B\to\infty$ 
    \begin{align}
        Q^{ML}(B) < Q^{SGD}(B).
    \end{align}
\end{theorem}
\begin{proof}
    Proof in Appendix \ref{app:mlmc_better}.
\end{proof}

~Theorem \ref{theo:MLMC_faster} implies that MLMC-SGD converges faster than standard SGD asymptotically when the cost of its gradient estimator, normalized by its variance bound, is lower than that of the single-fidelity estimator. 
In other words, the variance reduction achieved by MLMC must be substantial enough to justify its additional computational cost.
This condition holds practical significance when selecting a low-fidelity simulator to use in BISTRO. 
For BISTRO to outperform SGD asymptotically, the MLMC variance reduction of the gradient estimator must outweigh the per-iteration additional cost.
This is a necessary condition for selecting a low-fidelity simulator for BISTRO which can be verified by using a preliminary pilot study to test the quality of the low-fidelity simulator.

\subsection{Switching Condition}\label{sec:motivation}

In this section, we investigate and motivate the switching condition on Line \ref{line:switch} in Algorithm \ref{alg:BISTRO} under the strong convexity Assumption \ref{ass:hf_convex}. 
The MLMC-SGD convergence rate depends on the constant of convergence $\nu$ such that smaller constants will lead to faster convergence.
If the initial optimality gap $\mathbf{R}_H(x_1) - \mathbf{R}_*$ is large, the right-hand term in Equation \eqref{eq:opt_conv_const} will dominate the convergence rate. 
However, if the optimality gap is arbitrarily small, the left-hand term determines the rate of convergence.
Thus, the left-hand term corresponds to the smallest upper bound on the optimality gap.
By using an initialization phase with the trust-region optimizer, we aim to minimize $(\mathbf{R}_H(x_1) - \mathbf{R}_*)$ with respect to the MLMC-SGD starting position $x_1$ such that the left term dominates and the smallest upper bound is achieved. 
The smallest upper bound on the optimality gap is achieved when the left term dominates
\begin{align}
    \frac{2L_H W_{ML}}{c_H^2} &\geq \frac{2 L_H (W_{V,ML}\!+\!1)}{c_H}(\mathbf{R}_H(x_1)\!-\!\mathbf{R}_*) \implies
    \frac{W_{ML}}{c_H (W_{V,ML}\!+\!1)} \geq \mathbf{R}_H(x_1)\!-\!\mathbf{R}_*. 
    \label{eq:switching_condition}
\end{align}

We now compare this condition to the switching condition in BISTRO on Line \ref{line:switch} $\mathbf{\hat{R}}_{ML}^{(k+1)}(x_{k+1})>\mathbf{\hat{R}}_{ML}^{(k)}(x_{k})$.
The BISTRO switching condition is a surrogate for the theoretical condition in Equation \eqref{eq:switching_condition}.
When the BISTRO condition is satisfied, either the noise of the high-fidelity estimator is large, (and correlation needs to be exploited in MLMC-SGD), or the trust-region proposes a bad direction (and the low-fidelity curvature is no longer useful to exploit).
When the noise of the estimator is large, both Equation \eqref{eq:switching_condition} and Line \ref{line:switch} are easier to satisfy, justifying the BISTRO condition.

\section{Results}\label{sec:results}

BISTRO is now demonstrated on three problems. 
The first is a toy example of a quadratic objective with known oracle constants such as convexity and continuous differentiability constants. 
The toy example provides key insights on the convergence bounds of SGD and MLMC-SGD while demonstrating the utility of the switching condition.
The second is the Forretal example, taken from Ha and Mueller \cite{ha2024adaptive}, used to compare against the bi-fidelity trust-region algorithm ASTRO-BFDF with adaptive sampling rates.
The bi-fidelity extension of ASTRO-DF \cite{shashaani2018astro} is a derivative-free approach that learns when to switch between fidelities for cheaper convergence. 
The Forretal problem extends BISTRO to problems where the continuous differentiability and estimator noise constants are unknown, breaking theoretical guarantees but demonstrating empirical improvement over existing methods.
Finally, the third example is trajectory optimization problem that finds the optimal control to minimize the distance from a target location.
The space shuttle problem demonstrates BISTRO on a higher-dimensional example while demonstrating improvement over MLMC-SGD and BF-SVRG.
In all problems, we minimize the mean of the high-fidelity objective and aim to reduce the cost of achieving a solution close to the optimum.

\subsection{Quadratic Toy Problem}\label{sec:quadratic_problem}

Consider a simple quadratic objective function
\begin{align}
    J_H(x,\xi) = \sum_{i=1}^d x_i^2 + x_i\xi_i \quad\textrm{and}\quad J_L(x,\xi) &= \sum_{i=1}^d \left(\frac{x_i}{1.05}+1\right)^2 + \left(\frac{x_i}{1.05}+1\right)\xi_i,
\end{align} 
where $d=20$. 
The random variables are $\xi_i \sim \mathcal{N}(0, \sigma^2)$ where $\sigma^2 = 0.01$ for all $i\in\{1,\ldots,d\}$. 
The constants required for $\alpha$ and $\lambda$ were analytically and numerically calculated to be $L_H=c_H=2$, $L_L=c_L=1.81$, 
$W=0.2$, $W_{ML}=0.02$, $W_V=W_{V,ML}=1$, and $\mathbf{R}_*=0$.
The high- and low-fidelity samples per iteration to be $N=1$ and $K=M=10$.

The analytical mean of the objective function can be seen in Figure \ref{fig:quad-distance} on the y-axis for the optimization paths of each optimizer across the number of high-fidelity evaluations.
The convergence rate upper bound in \eqref{eq:rate_of_convergence}
is shown with $\nu=\frac{2L_H W_G}{c_H }(\mathbf{R}_H(x_1) - \mathbf{R}_*)$ in blue with the dashed-dotted line,  $\nu=\frac{2L_H W_{ML}}{c_H^2}$ in dashed green for MLMC-SGD, and $\nu=\frac{2L_H W}{c_H^2}$ in dashed blue for SGD.
The MLMC-SGD upper bound in green is smaller than the upper bound for SGD in blue since the MLMC estimator is able to achieve variance reduction.

\begin{figure}
    \centering
    \includegraphics[width=0.8\linewidth]{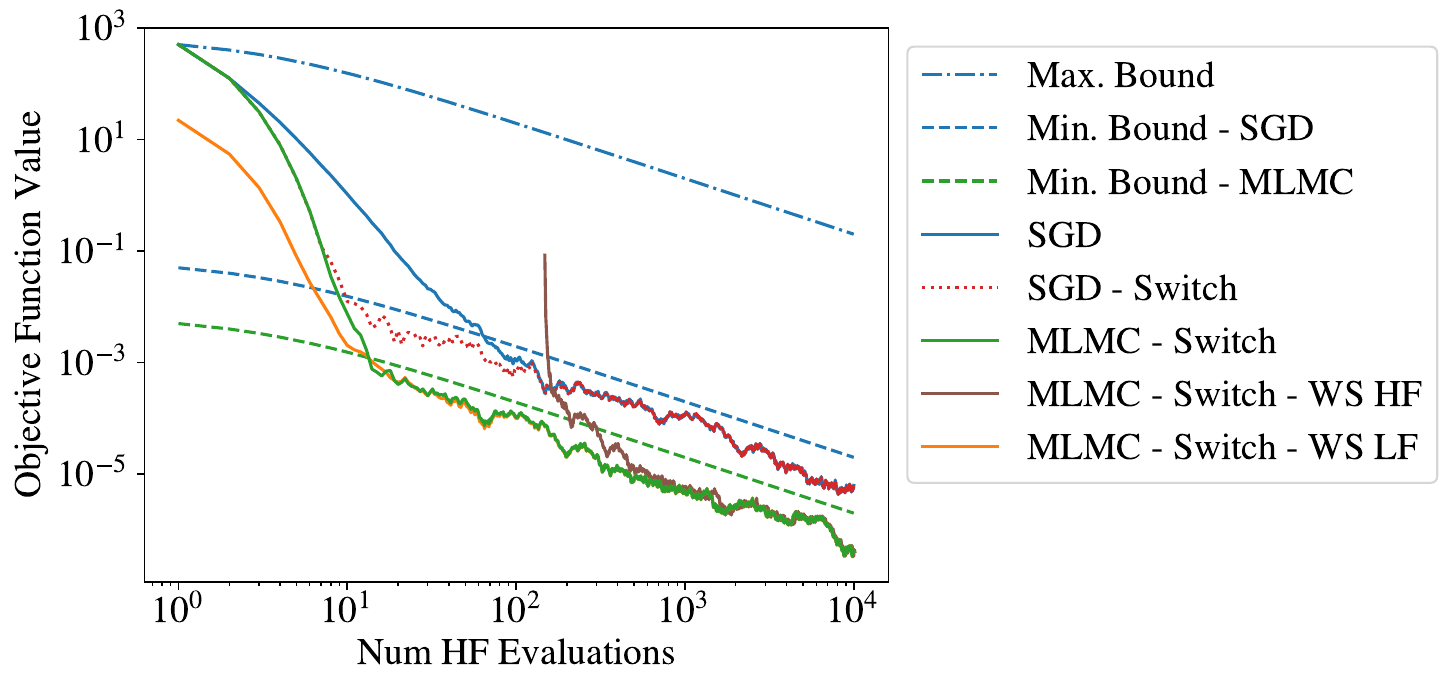}
    \caption{
    \textbf{Quadratic Example -} Analytic expected objective function value for each optimizer across number of high-fidelity evaluations. 
    The maximum and minimum optimality gap upper bounds (Eqs. \eqref{eq:rate_of_convergence}-\eqref{eq:opt_conv_const}) are also provided to validate theoretical results.
    }
    \label{fig:quad-distance}
\end{figure}

SGD is seen in blue where the initial optimality gap bounds the initial convergence until $\nu=\frac{2L_H W}{c_H^2}$ asymptotically dominates.
We also display SGD with the switching condition (Eq. \eqref{eq:switching_condition} with $W$ and $W_V$) in the dotted purple line as SGD-Switch.
The learning rate is held constant in SGD-Switch until the switching condition is met when the learning rate begins to decay.
Since SGD-Switch has a lower objective function value than SGD, this result demonstrates the importance of exploring the domain until an approximate solution is found before the learning rate should decay.
Without BISTRO, SGD still benefits from an exploration phase (constant learning rate) before exploiting the nearby solution (decaying learning rate).

Similarly, we show MLMC-SGD with the switching condition in green. 
The initial convergence is very similar to SGD-Switch and only when near the optimum does MLMC-SGD outperform SGD due to the reduced-noise gradient estimation. 
The use of MLMC-SGD only outperforms SGD asymptotically, and no benefit is seen compared to SGD before reaching an approximate solution.

Figure \ref{fig:quad-distance} also displays a warm-starting technique for MLMC-SGD. 
MLMC-SGD begins at a one-sample SAA solution of the LF (MLMC-Switch-WSLF) in orange and begins at the one-sample SAA solution of the HF in brown. 
Warm-starting with the high-fidelity solution is quite expensive and only converges to the asymptotic rate with 1,000 high-fidelity evaluations. 
Warm-starting with the low-fidelity is a much more practical solution, and offer some pre-asymptotic benefits.

The performance of BISTRO identically matches that of MLMC-SGD in this example.
Since this example is a quadratic objective whose Hessian is the identity matrix, the gradient vector points in the same direction as the trust-region solution. 
This example problem demonstrates the utility of variance-reduction approaches like MLMC and the utility of early exploration with the switching condition.
While the oracle constants are known in this example, the next section provides results for a more practical problem where the continuous differentiability, convexity, and other constants are unknown.

\subsection{Forretal Function}

We now compare BISTRO against the results published by ASTRO-BFDF \cite{ha2024adaptive} on the stochastic one-dimensional Forretal function.
Let the high- and low-fidelity functions be
\begin{align}
    J_H(x,\xi) &= (6x-2)^2\sin(12x-4) + \xi\sqrt{15+0.05x} \nonumber\\
    J_L^\kappa(x,\xi) &= (-2-\kappa^2+4\kappa)(6x-2)^2\sin(12x-4) + 10(x-0.5) - 5 + \xi\sqrt{15+0.05x} \nonumber
\end{align}
where $\kappa\in[0,1]$ is a parameter used to tune how similar the high- and low-fidelity functions are.
The high- and low-fidelity function costs are $1$ and $0.1$ respectively.
Since BISTRO requires the gradient of the functions, we estimate the gradient using finite-difference for a fair comparison with the gradient-free ASTRO-DF and ASTRO-BFDF approaches.
The noise terms are distributed according to $\xi\sim\mathcal{N}(0,1)$.

The total optimization budget for each algorithm is 300.
The optimization hyperparameters for ASTRO-DF and ASTRO-BFDF can be found in Appendix D of \cite{ha2024adaptive}.
BISTRO's hyperparameters are set to $N=K=1$, $M=2$, $\alpha_0=0.001$, $\gamma=2000$, and $\Delta=0.5$.
The number of high- and low-fidelity samples in the MLMC estimator were chosen to be small to reduce the per-iteration cost.
The initial trust-region size can be set as a maximum step size of the algorithm, but was tuned in this example for best performance.
The initial learning rate and decay parameter $\gamma$ were tuned where $\beta$ is implicitly defined by $\alpha_0$.
No practical advice is offered to initialize $\alpha_0$ and $\beta$, since this issue is well-known in SGD methods \cite{xu2019learning}.

\begin{figure}
    \centering
    \includegraphics[width=0.49\linewidth]{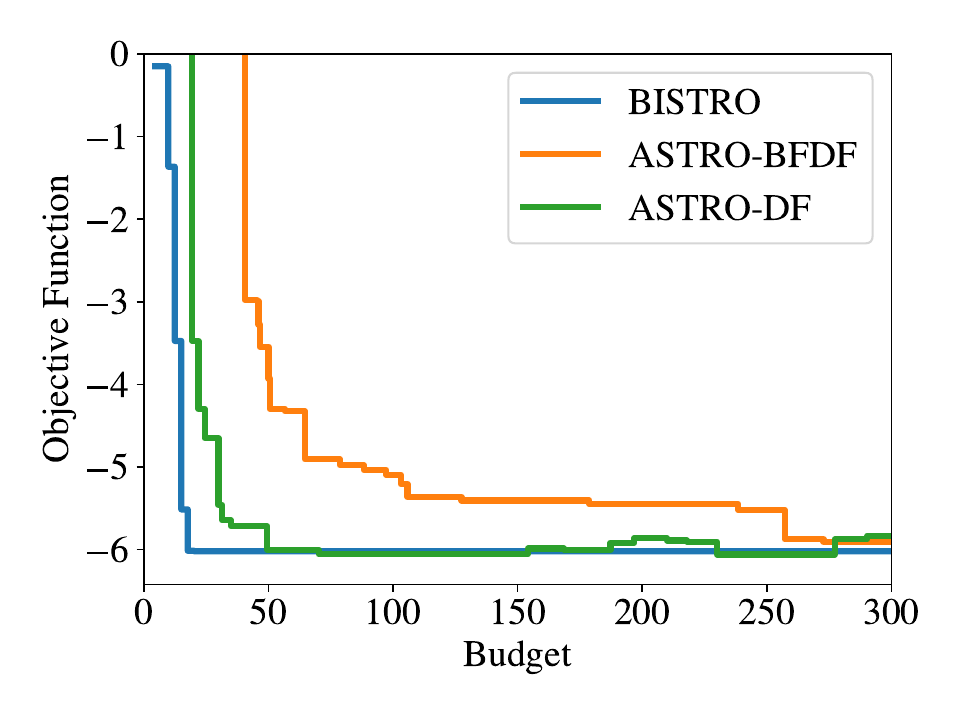}
    \includegraphics[width=0.49\linewidth]{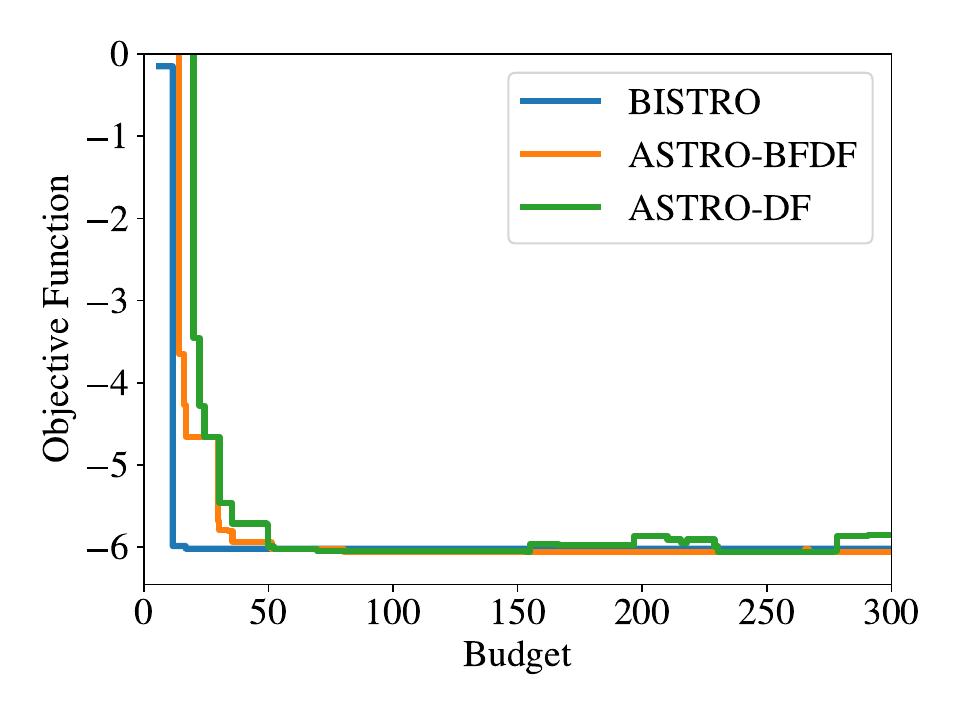}
    \caption{
    \textbf{Forretal Function -} (\textit{Left}) Mean objective function value for $\kappa=0.1$ across 20 runs for each method.
    (\textit{Right}) Mean objective function value for $\kappa=0.9$ across 20 runs for each method.
    }
    \label{fig:forretal}
\end{figure}

The analytical mean of the objective function of ASTRO-DF \cite{shashaani2018astro}, ASTRO-BFDF \cite{ha2024adaptive}, and BISTRO are seen in Figure \ref{fig:forretal} with two realizations of $\kappa=\{0.1,0.9\}$.
Each algorithm starts at $x=0.6$ (Y. Ha, personal correspondence, Nov. 2025) and eventually converges to the minimum.
BISTRO is able to reach the minimum faster than ASTRO-DF and ASTRO-BFDF in both examples.
By using the trust-region to warm-start a reduced-variance SGD algorithm, BISTRO mitigates the cost of the expensive trust-region optimizer.
Both ASTRO algorithms require using a trust-region all the way to the solution, while BISTRO uses SGD asymptotically. 
When $\kappa=0.1$, BISTRO is able to achieve the minimum objective function value 2.8 times faster than ASTRO-DF and 28.4 times faster than ASTRO-BFDF.
When $\kappa=0.9$, BISTRO is able to achieve the minimum objective function 3 times faster than ASTRO-DF and ASTRO-BFDF. 
Since ASTRO-DF and ASTRO-BFDF are adaptive sampling approaches with increasing per-iteration costs, this example demonstrates that decaying learning rate approaches (like BISTRO) may be better suited for problems with low budgets.

\subsection{Space Shuttle Reentry}
We now test BISTRO on a 3 degree-of-freedom simulator of the space shuttle from Chapter 6 in Betts \cite{betts2010practical}. 
Over a 2,000 second flight, we find the angle-of-attack and bank angle trajectories $(\eta,\mu)$ that minimize the distance from a target while penalizing changes in control
\begin{align}
    J_H(x,\xi) = ||(\phi_{f}(x,\xi),\theta_{f}(x,\xi)) - (55^\circ,25^\circ)||_2 + \left|\left|\frac{d}{dt} x\right|\right|,
\end{align}
where the longitude and latitude are ($\phi,\theta$), the control is the angle-of-attack and bank-angle trajectories $x=\{(\eta_t,\mu_t): t\in[0,T]\}$, and $T=2,000$s. 
A direct single-shooting approach \cite{betts2010practical} is used to avoid introducing the states as optimization variables, and the optimization variables are bounded after each iteration to maintain realistic simulations.
We model the control trajectories with a cubic spline interpolation of 10 nodes for each signal such that we have 20 nodes to optimize over to find $x^*\in\mathbb{R}^{20}$.

To introduce uncertainty, the lift and drag coefficients are assumed to be uncertain with a normal distribution of 10\% standard deviation from the percent difference of the nominal values from Betts \cite{betts2010practical}.

The high- and low-fidelity simulators use a fourth-order Runge Kutta scheme where the high- and low-fidelity use time steps of 0.2 seconds and 20 seconds, respectively.
The high- and low-fidelity objective functions take $1.70\times10^{-2}$ and $3.46\times10^{-4}$
seconds to evaluate while the gradient of the high- and low-fidelity objective function takes $8.56\times10^{-1}$ and $9.55\times10^{-3}$ seconds.
Each optimizer was initialized with $(\eta_t,\mu_t)=(19^\circ,-40^\circ)$ for all $t\in\{1,\ldots,20\}$.
BISTRO's hyperparameters are set to $N=1$, $M=K=10$, $\alpha_0=0.01$, $\gamma=100$, and $\Delta=100$.

BISTRO, SGD, MLMC-SGD, and BF-SVRG
were evaluated across 10 random seeds with the 20-50-80 
quantiles of the objective function shown in the left plot of Figure \ref{fig:traj-objs} across the total budgets.
The objective functions were calculated with 100 reference samples.
BISTRO is immediately able to find lower objective function values compared to all other methods and plateaus quickly.
Unlike the previous examples, the objective function value is not expected to converge to zero, and the optimal objective function value is unknown.
BISTRO is able to immediately identify a location with an objective function value matching previous works.
It takes MLMC-SGD, SGD, and BF-SVRG 29 times longer to achieve the similar objective function values as BISTRO. 
The near global search of the trust-region is able to quickly find promising directions for the high-fidelity optimizer to explore, as long as the low-fidelity function is informative. 

\begin{figure}
    \centering
    \includegraphics[width=0.45\linewidth]{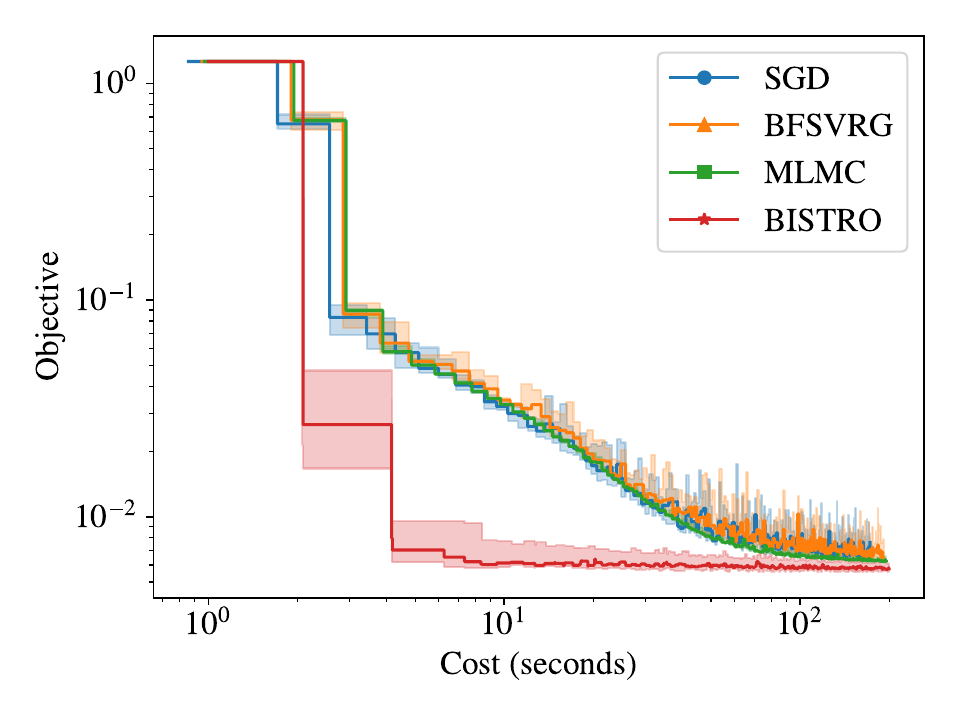}
    \includegraphics[width=0.45\linewidth]{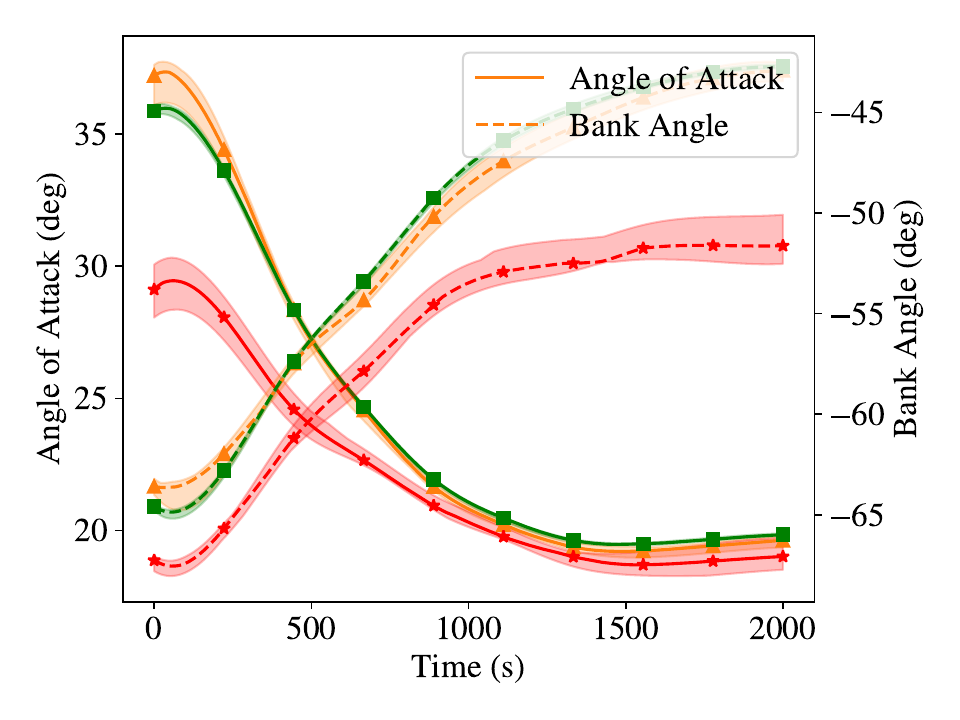}
    \caption{\textbf{Trajectory Optimization -} (\textit{Left}) The objective function values across the budget of the optimizers. (\textit{Right}) The optimal controls found using BF-SVRG, MLMC-SGD, and BISTRO.}
    \label{fig:traj-objs}
\end{figure}

The right plot of Figure \ref{fig:traj-objs} displays the final controls learned by MLMC-SGD, BISTRO, and BF-SVRG. 
The nodes are indicated by the 10 points and the 20-50-80 quantiles are shown in the bands.
Notably, BISTRO did not converge to the same control trajectories as BF-SVRG and MLMC. 
BISTRO achieved a lower final objective function value by finding a separate local minimum through the trust-region exploration phase.

\FloatBarrier
\section{CONCLUSION}

This paper introduced BISTRO, a novel bi-fidelity optimization under uncertainty algorithm that simultaneously exploits both curvature information in the design space and correlation structure in the random space.
By combining a trust-region approach with control variate-based variance reduction, BISTRO addresses a fundamental limitation of existing methods that typically leverage only one of these two complementary sources of information.

Our theoretical analysis demonstrates that BISTRO maintains the same asymptotic convergence rates as reduced-variance stochastic gradient descent while offering dramatically improved finite-sample performance.
This two-phase approach overcomes the key limitation of traditional trust-region methods, which become inefficient near the optimum where low-fidelity curvature information loses its utility.

The numerical experiments validate both the theoretical properties and practical benefits of BISTRO.
On the 20-dimensional space shuttle reentry optimization problem, BISTRO achieved solutions 29 times faster than existing single- and bi-fidelity stochastic gradient methods.
These results demonstrate that joint exploitation of both design-space curvature and random-space correlation can yield substantial computational savings in realistic engineering applications.

Several avenues for future research emerge from this work.
First, extending the multi-level Monte Carlo estimator to more optimal approximate control variate methods could further improve variance reduction.
Since BISTRO can utilize any unbiased estimator, future implementations could incorporate robust optimization objectives such as probability of failure or conditional value at risk.
Second, we would like to extend the ideas presented here to handle constraints, representing an important direction for practical applications.

\section{ACKNOWLEDGMENTS}
This work was supported by a NASA subcontract 
at the University of Michigan titled, ``Novel variance reduction approaches for multifidelity uncertainty quantification in high-fidelity Entry, Descent, and Landing" under grant number 80NSSC22K1007.

\bibliographystyle{siamplain}
\bibliography{references}

@article{korondi2021multi,
  title={{Multi-fidelity design optimisation strategy under uncertainty with limited computational budget}},
  author={Korondi, P{\'e}ter Z{\'e}n{\'o} and Marchi, Mariapia and Parussini, Lucia and Poloni, Carlo},
  journal={Optimization and Engineering},
  volume={22},
  number={2},
  pages={1039--1064},
  year={2021},
  publisher={Springer}
}

@article{ng2014multifidelity,
  title={{Multifidelity approaches for optimization under uncertainty}},
  author={Ng, Leo WT and Willcox, Karen E},
  journal={International Journal for numerical methods in Engineering},
  volume={100},
  number={10},
  pages={746--772},
  year={2014},
  publisher={Wiley Online Library}
}

@article{hart2023hyper,
  title={{Hyper-differential sensitivity analysis with respect to model discrepancy: optimal solution updating}},
  author={Hart, Joseph and van Bloemen Waanders, Bart},
  journal={Computer Methods in Applied Mechanics and Engineering},
  volume={412},
  pages={116082},
  year={2023},
  publisher={Elsevier}
}

@article{alexandrov1998trust,
  title={{A trust-region framework for managing the use of approximation models in optimization}},
  author={Alexandrov, Natalia M and Dennis Jr, John E and Lewis, Robert Michael and Torczon, Virginia},
  journal={Structural optimization},
  volume={15},
  number={1},
  pages={16--23},
  year={1998},
  publisher={Springer}
}

@article{march2012provably,
  title={{Provably convergent multifidelity optimization algorithm not requiring high-fidelity derivatives}},
  author={March, Andrew and Willcox, Karen},
  journal={AIAA journal},
  volume={50},
  number={5},
  pages={1079--1089},
  year={2012}
}

@article{kouri2013trust,
  title={{A trust-region algorithm with adaptive stochastic collocation for PDE optimization under uncertainty}},
  author={Kouri, Drew P and Heinkenschloss, Matthias and Ridzal, Denis and van Bloemen Waanders, Bart G},
  journal={SIAM Journal on Scientific Computing},
  volume={35},
  number={4},
  pages={A1847--A1879},
  year={2013},
  publisher={SIAM}
}

@article{menhorn2024multilevel,
  title={{Multilevel Monte Carlo estimators for derivative-free optimization under uncertainty}},
  author={Menhorn, Friedrich and Geraci, Gianluca and Seidl, D Thomas and Marzouk, Youssef M and Eldred, Michael S and Bungartz, Hans-Joachim},
  journal={International Journal for Uncertainty Quantification},
  volume={14},
  number={3},
  year={2024},
  publisher={Begel House Inc.}
}

@inproceedings{eldred2004second,
  title={{Second-order corrections for surrogate-based optimization with model hierarchies}},
  author={Eldred, Michael and Giunta, Anthony and Collis, Samuel},
  booktitle={10th AIAA/ISSMO multidisciplinary analysis and optimization conference},
  pages={4457},
  year={2004}
}

@article{robinson2008surrogate,
  title={{Surrogate-based optimization using multifidelity models with variable parameterization and corrected space mapping}},
  author={Robinson, TD and Eldred, Michael S and Willcox, Karen E and Haimes, Robert},
  journal={AIAA journal},
  volume={46},
  number={11},
  pages={2814--2822},
  year={2008}
}

@article{march2011gradient,
  title={{Gradient-based multifidelity optimisation for aircraft design using Bayesian model calibration}},
  author={March, Andrew and Willcox, Karen and Wang, Qiqi},
  journal={The Aeronautical Journal},
  volume={115},
  number={1174},
  pages={729--738},
  year={2011},
  publisher={Cambridge University Press}
}

@article{wang2013variance,
  title={{Variance reduction for stochastic gradient optimization}},
  author={Wang, Chong and Chen, Xi and Smola, Alexander J and Xing, Eric P},
  journal={Advances in neural information processing systems},
  volume={26},
  year={2013}
}

@article{shapiro1991asymptotic,
  title={{Asymptotic analysis of stochastic programs}},
  author={Shapiro, Alexander},
  journal={Annals of Operations Research},
  volume={30},
  pages={169--186},
  year={1991},
  publisher={Springer}
}

@incollection{homem2014stochastic,
  title={{Stochastic constraints and variance reduction techniques}},
  author={Homem-de-Mello, Tito and Bayraksan, G{\"u}zin},
  booktitle={Handbook of simulation optimization},
  pages={245--276},
  year={2014},
  publisher={Springer}
}

@article{agrawal2023multi,
  title={{Multi-fidelity constrained optimization for stochastic black box simulators}},
  author={Agrawal, Atul and Ravi, Kislaya and Koutsourelakis, Phaedon-Stelios and Bungartz, Hans-Joachim},
  journal={arXiv preprint arXiv:2311.15137},
  year={2023}
}

@article{pellegrini2023multi,
  title={{A multi-fidelity active learning method for global design optimization problems with noisy evaluations}},
  author={Pellegrini, Riccardo and Wackers, Jeroen and Broglia, Riccardo and Serani, Andrea and Visonneau, Michel and Diez, Matteo},
  journal={Engineering with Computers},
  volume={39},
  number={5},
  pages={3183--3206},
  year={2023},
  publisher={Springer}
}

@article{forrester2007multi,
  title={{Multi-fidelity optimization via surrogate modelling}},
  author={Forrester, Alexander IJ and S{\'o}bester, Andr{\'a}s and Keane, Andy J},
  journal={Proceedings of the royal society a: mathematical, physical and engineering sciences},
  volume={463},
  number={2088},
  pages={3251--3269},
  year={2007},
  publisher={The Royal Society London}
}

@article{bryson2018multifidelity,
  title={{Multifidelity quasi-Newton method for design optimization}},
  author={Bryson, Dean E and Rumpfkeil, Markus P},
  journal={AIAA Journal},
  volume={56},
  number={10},
  pages={4074--4086},
  year={2018},
  publisher={American Institute of Aeronautics and Astronautics}
}

@article{pellegrini2022derivative,
  title={{A derivative-free line-search algorithm for simulation-driven design optimization using multi-fidelity computations}},
  author={Pellegrini, Riccardo and Serani, Andrea and Liuzzi, Giampaolo and Rinaldi, Francesco and Lucidi, Stefano and Diez, Matteo},
  journal={Mathematics},
  volume={10},
  number={3},
  pages={481},
  year={2022},
  publisher={MDPI}
}

@article{ha2025multi,
  title={{Multi-fidelity stochastic trust region method with adaptive sampling}},
  author={Ha, Yunsoo and Mueller, Juliane},
  journal={arXiv preprint arXiv:2508.03901},
  year={2025}
}

@article{ha2024adaptive,
  title={{Adaptive sampling-based bi-fidelity stochastic trust region method for derivative-free stochastic optimization}},
  author={Ha, Yunsoo and Mueller, Juliane},
  journal={arXiv preprint arXiv:2408.04625},
  year={2024}
}

@article{wu2022gradient,
  title={{A gradient-based sequential multifidelity approach to multidisciplinary design optimization}},
  author={Wu, Neil and Mader, Charles A and Martins, Joaquim RRA},
  journal={Structural and Multidisciplinary Optimization},
  volume={65},
  number={4},
  pages={131},
  year={2022},
  publisher={Springer}
}

@article{shah2015multi,
  title={{Multi-fidelity robust aerodynamic design optimization under mixed uncertainty}},
  author={Shah, Harsheel and Hosder, Serhat and Koziel, Slawomir and Tesfahunegn, Yonatan A and Leifsson, Leifur},
  journal={Aerospace Science and Technology},
  volume={45},
  pages={17--29},
  year={2015},
  publisher={Elsevier}
}

@article{de2020bi,
  title={{Bi-fidelity stochastic gradient descent for structural optimization under uncertainty}},
  author={De, Subhayan and Maute, Kurt and Doostan, Alireza},
  journal={Computational Mechanics},
  volume={66},
  pages={745--771},
  year={2020},
  publisher={Springer}
}

@book{betts2010practical,
  title={{Practical methods for optimal control and estimation using nonlinear programming}},
  author={Betts, John T},
  year={2010},
  publisher={SIAM}
}

@article{bottou2018optimization,
  title={{Optimization methods for large-scale machine learning}},
  author={Bottou, L{\'e}on and Curtis, Frank E and Nocedal, Jorge},
  journal={SIAM review},
  volume={60},
  number={2},
  pages={223--311},
  year={2018},
  publisher={SIAM}
}

@article{shashaani2018astro,
  title={{ASTRO-DF: A class of adaptive sampling trust-region algorithms for derivative-free stochastic optimization}},
  author={Shashaani, Sara and Hashemi, Fatemeh S and Pasupathy, Raghu},
  journal={SIAM Journal on Optimization},
  volume={28},
  number={4},
  pages={3145--3176},
  year={2018},
  publisher={SIAM}
}

@book{fu2015handbook,
  title={{Handbook of simulation optimization}},
  author={Fu, Michael C and others},
  volume={216},
  year={2015},
  publisher={Springer}
}

@article{giles2008multilevel,
  title={{Multilevel Monte Carlo path simulation}},
  author={Giles, Michael B},
  journal={Operations research},
  volume={56},
  number={3},
  pages={607--617},
  year={2008},
  publisher={INFORMS}
}

@article{dixon2024covariance,
  title={{Covariance expressions for multifidelity sampling with multioutput, multistatistic estimators: application to approximate control variates}},
  author={Dixon, Thomas O and Warner, James E and Bomarito, Geoffrey F and Gorodetsky, Alex A},
  journal={SIAM/ASA Journal on Uncertainty Quantification},
  volume={12},
  number={3},
  pages={1005--1049},
  year={2024},
  publisher={SIAM}
}

@article{roald2023power,
  title={{Power systems optimization under uncertainty: A review of methods and applications}},
  author={Roald, Line A and Pozo, David and Papavasiliou, Anthony and Molzahn, Daniel K and Kazempour, Jalal and Conejo, Antonio},
  journal={Electric Power Systems Research},
  volume={214},
  pages={108725},
  year={2023},
  publisher={Elsevier}
}

@article{fehrman2020convergence,
  title={{Convergence rates for the stochastic gradient descent method for non-convex objective functions}},
  author={Fehrman, Benjamin and Gess, Benjamin and Jentzen, Arnulf},
  journal={Journal of Machine Learning Research},
  volume={21},
  number={136},
  pages={1--48},
  year={2020}
}

@article{frikha2016multi,
  title={{Multi-level stochastic approximation algorithms}},
  author={Frikha, Noufel},
  year={2016}
}

@article{pasupathy2018sampling,
  title={{On sampling rates in simulation-based recursions}},
  author={Pasupathy, Raghu and Glynn, Peter and Ghosh, Soumyadip and Hashemi, Fatemeh S},
  journal={SIAM Journal on Optimization},
  volume={28},
  number={1},
  pages={45--73},
  year={2018},
  publisher={SIAM}
}

@article{byrd2012sample,
  title={{Sample size selection in optimization methods for machine learning}},
  author={Byrd, Richard H and Chin, Gillian M and Nocedal, Jorge and Wu, Yuchen},
  journal={Mathematical programming},
  volume={134},
  number={1},
  pages={127--155},
  year={2012},
  publisher={Springer}
}

@article{chen2018stochastic,
  title={{Stochastic optimization using a trust-region method and random models}},
  author={Chen, Ruobing and Menickelly, Matt and Scheinberg, Katya},
  journal={Mathematical Programming},
  volume={169},
  pages={447--487},
  year={2018},
  publisher={Springer}
}

@article{curtis2019stochastic,
  title={{A stochastic trust region algorithm based on careful step normalization}},
  author={Curtis, Frank E and Scheinberg, Katya and Shi, Rui},
  journal={Informs Journal on Optimization},
  volume={1},
  number={3},
  pages={200--220},
  year={2019},
  publisher={INFORMS}
}

@article{blanchet2019convergence,
  title={{Convergence rate analysis of a stochastic trust-region method via supermartingales}},
  author={Blanchet, Jose and Cartis, Coralia and Menickelly, Matt and Scheinberg, Katya},
  journal={INFORMS journal on optimization},
  volume={1},
  number={2},
  pages={92--119},
  year={2019},
  publisher={INFORMS}
}

@article{peherstorfer2018survey,
  title={{Survey of multifidelity methods in uncertainty propagation, inference, and optimization}},
  author={Peherstorfer, Benjamin and Willcox, Karen and Gunzburger, Max},
  journal={Siam Review},
  volume={60},
  number={3},
  pages={550--591},
  year={2018},
  publisher={SIAM}
}

@article{tong2022optimization,
  title={{Optimization under rare chance constraints}},
  author={Tong, Shanyin and Subramanyam, Anirudh and Rao, Vishwas},
  journal={SIAM Journal on Optimization},
  volume={32},
  number={2},
  pages={930--958},
  year={2022},
  publisher={SIAM}
}

@article{gorodetsky2020generalized,
  title={{A generalized approximate control variate framework for multifidelity uncertainty quantification}},
  author={Gorodetsky, Alex A and Geraci, Gianluca and Eldred, Michael S and Jakeman, John D},
  journal={Journal of Computational Physics},
  volume={408},
  pages={109257},
  year={2020},
  publisher={Elsevier}
}

@article{babcock2024multi,
  title={{Multi-fidelity error-estimate-based model management}},
  author={Babcock, Tucker and Hall, Dustin and Gray, Justin S and Hicken, Jason E},
  journal={Structural and Multidisciplinary Optimization},
  volume={67},
  number={3},
  pages={36},
  year={2024},
  publisher={Springer}
}

@article{peherstorfer2016optimal,
  title={{Optimal model management for multifidelity Monte Carlo estimation}},
  author={Peherstorfer, Benjamin and Willcox, Karen and Gunzburger, Max},
  journal={SIAM Journal on Scientific Computing},
  volume={38},
  number={5},
  pages={A3163--A3194},
  year={2016},
  publisher={SIAM}
}

@article{toscano2022bayesian,
  title={{Bayesian optimization with expensive integrands}},
  author={Toscano-Palmerin, Saul and Frazier, Peter I},
  journal={SIAM Journal on Optimization},
  volume={32},
  number={2},
  pages={417--444},
  year={2022},
  publisher={SIAM}
}

@article{alarie2021optimization,
  title={{Optimization of stochastic blackboxes with adaptive precision}},
  author={Alarie, St{\'e}phane and Audet, Charles and Bouchet, Pierre-Yves and Digabel, Sébastien Le},
  journal={SIAM Journal on Optimization},
  volume={31},
  number={4},
  pages={3127--3156},
  year={2021},
  publisher={SIAM}
}

@article{hannah2014semiconvex,
  title={{Semiconvex regression for metamodeling-based optimization}},
  author={Hannah, Lauren A and Powell, Warren B and Dunson, David B},
  journal={SIAM Journal on Optimization},
  volume={24},
  number={2},
  pages={573--597},
  year={2014},
  publisher={SIAM}
}

@article{glowinski1995simulation,
  title={{On the simulation and control of some friction constrained motions}},
  author={Glowinski, Roland and Kearsley, Anthony J},
  journal={SIAM Journal on Optimization},
  volume={5},
  number={3},
  pages={681--694},
  year={1995},
  publisher={SIAM}
}

@article{yin2002recursive,
  title={{Recursive algorithms for stock liquidation: a stochastic optimization approach}},
  author={Yin, G and Liu, RH and Zhang, Qing},
  journal={SIAM Journal on Optimization},
  volume={13},
  number={1},
  pages={240--263},
  year={2002},
  publisher={SIAM}
}

@article{ermoliev2013sample,
  title={{Sample average approximation method for compound stochastic optimization problems}},
  author={Ermoliev, Yuri M and Norkin, Vladimir I},
  journal={SIAM Journal on Optimization},
  volume={23},
  number={4},
  pages={2231--2263},
  year={2013},
  publisher={SIAM}
}

@inproceedings{eriksson2021high,
  title={{High-dimensional Bayesian optimization with sparse axis-aligned subspaces}},
  author={Eriksson, David and Jankowiak, Martin},
  booktitle={Uncertainty in Artificial Intelligence},
  pages={493--503},
  year={2021},
  organization={PMLR}
}

@article{binois2022survey,
  title={{A survey on high-dimensional Gaussian process modeling with application to Bayesian optimization}},
  author={Binois, Mickael and Wycoff, Nathan},
  journal={ACM Transactions on Evolutionary Learning and Optimization},
  volume={2},
  number={2},
  pages={1--26},
  year={2022},
  publisher={ACM New York, NY}
}

@article{gorodetsky2024grouped,
  title={{Grouped approximate control variate estimators}},
  author={Gorodetsky, Alex A and Jakeman, John D and Eldred, Michael S},
  journal={arXiv preprint arXiv:2402.14736},
  year={2024}
}

@article{qian2018multifidelity,
  title={{Multifidelity Monte Carlo estimation of variance and sensitivity indices}},
  author={Qian, Elizabeth and Peherstorfer, Benjamin and O'Malley, Daniel and Vesselinov, Velimir V and Willcox, Karen},
  journal={SIAM/ASA Journal on Uncertainty Quantification},
  volume={6},
  number={2},
  pages={683--706},
  year={2018},
  publisher={SIAM}
}

@article{ayoul2023quantifying,
  title={{ {Quantifying uncertain system outputs via the multi-level Monte Carlo method- distribution and robustness measures}}},
  author={Ayoul-Guilmard, Quentin and Ganesh, Sundar and Krumscheid, Sebastian and Nobile, Fabio},
  journal={International Journal for Uncertainty Quantification},
  volume={13},
  number={5},
  year={2023},
  publisher={Begel House Inc.}
}

@article{peherstorfer2016multifidelity,
  title={{Multifidelity importance sampling}},
  author={Peherstorfer, Benjamin and Cui, Tiangang and Marzouk, Youssef and Willcox, Karen},
  journal={Computer Methods in Applied Mechanics and Engineering},
  volume={300},
  pages={490--509},
  year={2016},
  publisher={Elsevier}
}

@article{qiu2015robust,
  title={{Robust portfolio optimization}},
  author={Qiu, Huitong and Han, Fang and Liu, Han and Caffo, Brian},
  journal={Advances in Neural Information Processing Systems},
  volume={28},
  year={2015}
}

@inproceedings{chaudhuri2020risk,
  title={{Risk-based design optimization via probability of failure, conditional value-at-risk, and buffered probability of failure}},
  author={Chaudhuri, Anirban and Norton, Matthew and Kramer, Boris},
  booktitle={AIAA Scitech 2020 Forum},
  pages={2130},
  year={2020}
}

@inproceedings{eldred2002formulations,
  title={{Formulations for surrogate-based optimization under uncertainty}},
  author={Eldred, Michael and Giunta, Anthony and Wojtkiewicz, Steven and Trucano, Timothy},
  booktitle={9th AIAA/ISSMO symposium on multidisciplinary analysis and optimization},
  pages={5585},
  year={2002}
}

@article{acar2021modeling,
  title={{Modeling, analysis, and optimization under uncertainties: a review}},
  author={Acar, Erdem and Bayrak, Gamze and Jung, Yongsu and Lee, Ikjin and Ramu, Palaniappan and Ravichandran, Suja Shree},
  journal={Structural and Multidisciplinary Optimization},
  volume={64},
  number={5},
  pages={2909--2945},
  year={2021},
  publisher={Springer}
}

@article{zhao2018data,
  title={{Data-driven risk-averse stochastic optimization with Wasserstein metric}},
  author={Zhao, Chaoyue and Guan, Yongpei},
  journal={Operations Research Letters},
  volume={46},
  number={2},
  pages={262--267},
  year={2018},
  publisher={Elsevier}
}

@article{rodriguez1998convergence,
  title={{Convergence of trust region augmented Lagrangian methods using variable fidelity approximation data}},
  author={Rodriguez, Jose F and Renaud, John E and Watson, Layne Terry},
  journal={Structural optimization},
  volume={15},
  number={3},
  pages={141--156},
  year={1998},
  publisher={Springer}
}

@article{do2023multi,
  title={{Multi-fidelity Bayesian optimization in engineering design}},
  author={Do, Bach and Zhang, Ruda},
  journal={arXiv preprint arXiv:2311.13050},
  year={2023}
}

@article{nelson1990control,
  title={{Control variate remedies}},
  author={Nelson, Barry L},
  journal={Operations Research},
  volume={38},
  number={6},
  pages={974--992},
  year={1990},
  publisher={INFORMS}
}

@article{xu2019learning,
  title={{Learning an adaptive learning rate schedule}},
  author={Xu, Zhen and Dai, Andrew M and Kemp, Jonas and Metz, Luke},
  journal={arXiv preprint arXiv:1909.09712},
  year={2019}
}

@article{deng2009variable,
  title={{Variable-number sample-path optimization}},
  author={Deng, Geng and Ferris, Michael C},
  journal={Mathematical Programming},
  volume={117},
  number={1},
  pages={81--109},
  year={2009},
  publisher={Springer}
}

@inproceedings{loizou2021stochastic,
  title={{Stochastic polyak step-size for SGD: An adaptive learning rate for fast convergence}},
  author={Loizou, Nicolas and Vaswani, Sharan and Laradji, Issam Hadj and Lacoste-Julien, Simon},
  booktitle={International Conference on Artificial Intelligence and Statistics},
  pages={1306--1314},
  year={2021},
  organization={PMLR}
}

@article{kingma2014adam,
  title={{Adam: A method for stochastic optimization}},
  author={Kingma, Diederik P},
  journal={arXiv preprint arXiv:1412.6980},
  year={2014}
}

@article{ge2019step,
  title={{The step decay schedule: A near optimal, geometrically decaying learning rate procedure for least squares}},
  author={Ge, Rong and Kakade, Sham M and Kidambi, Rahul and Netrapalli, Praneeth},
  journal={Advances in neural information processing systems},
  volume={32},
  year={2019}
}

@article{van2021mg,
  title={{MG/OPT and multilevel Monte Carlo for robust optimization of PDEs}},
  author={Van Barel, Andreas and Vandewalle, Stefan},
  journal={SIAM Journal on Optimization},
  volume={31},
  number={3},
  pages={1850--1876},
  year={2021},
  publisher={SIAM}
}

@article{kouri2014multilevel,
  title={{A multilevel stochastic collocation algorithm for optimization of PDEs with uncertain coefficients}},
  author={Kouri, Drew P},
  journal={SIAM/ASA Journal on Uncertainty Quantification},
  volume={2},
  number={1},
  pages={55--81},
  year={2014},
  publisher={SIAM}
}

@inproceedings{lewis2000multigrid,
  title={{A multigrid approach to the optimization of systems governed by differential equations}},
  author={Lewis, Robert and Nash, Stephen},
  booktitle={8th symposium on multidisciplinary analysis and optimization},
  pages={4890},
  year={2000}
}

@book{nocedal2006numerical,
  title={{Numerical optimization}},
  author={Nocedal, Jorge and Wright, Stephen J},
  year={2006},
  publisher={Springer}
}

@book{boyd2004convex,
  title={{Convex optimization}},
  author={Boyd, Stephen and Vandenberghe, Lieven},
  year={2004},
  publisher={Cambridge university press}
}

@article{dereich2019general,
  title={{General multilevel adaptations for stochastic approximation algorithms of Robbins--Monro and Polyak--Ruppert type}},
  author={Dereich, Steffen and M{\"u}ller-Gronbach, Thomas},
  journal={Numerische Mathematik},
  volume={142},
  number={2},
  pages={279--328},
  year={2019},
  publisher={Springer}
}

@book{asmussen2007stochastic,
  title={{Stochastic simulation: algorithms and analysis}},
  author={Asmussen, S{\o}ren and Glynn, Peter W},
  volume={57},
  year={2007},
  publisher={Springer}
}

@article{schmidt2017minimizing,
  title={{Minimizing finite sums with the stochastic average gradient}},
  author={Schmidt, Mark and Le Roux, Nicolas and Bach, Francis},
  journal={Mathematical Programming},
  volume={162},
  number={1},
  pages={83--112},
  year={2017},
  publisher={Springer}
}

@article{johnson2013accelerating,
  title={{Accelerating stochastic gradient descent using predictive variance reduction}},
  author={Johnson, Rie and Zhang, Tong},
  journal={Advances in neural information processing systems},
  volume={26},
  year={2013}
}

\appendix

\section{Appendix - Optimal MLMC-SGD parameters}\label{app:opt_learning_rate}

According to Bottou et al. \cite{bottou2018optimization}, the constant of convergence is $\nu = \frac{\beta^2 L_H W}{2(\beta c_H-1)} $
asymptotically as $k\to \infty$, suggesting that the initial optimality gap (right term in Equation \eqref{eq:conv_constant}) does not impact the asymptotic convergence.
Thus, to minimize the asymptotic upper bound, we minimize $\nu$ by tuning the learning rate parameters $\beta$ and $\gamma$.
With respect to $\beta$, the constant $\nu$ is minimized at $\beta=\frac{2}{c_H}$. 
Since $\gamma$ is independent of $\nu$, it is a free variable that we are able to choose. 
In practice, maximizing the initial learning rate $\alpha_0$ is common to allow larger steps in early iterations for more exploration.
Thus, we set the initial learning rate to its upper limit $\alpha_0=\frac{1}{L_H W_G}$.
Therefore, we use $\alpha=\beta/(\gamma+k)$ to solve for the parameter $\gamma= 2\frac{L_H W_G}{c}-1$.

\section{Appendix - MLMC-SGD Converges Faster than SGD} \label{app:mlmc_better}

We find when the upper bound on the optimality gap for MLMC-SGD is smaller asymptotically than SGD.
First, we must show when the upper bound for MLMC-SGD is smaller than SGD at a given budget $B$.

\begin{theorem}[MLMC-SGD converges faster than SGD at a given budget]
\label{theo:mlmc_sgd}
    Let $W_{G,ML}=W_{V,ML}+1$ and $W_G=W_V+1$.
    Under Assumptions \ref{ass:hf_smooth_convex} \ref{ass:obj_bound}, \ref{ass:var_bound},\ref{ass:var_bound_ml}, and \ref{ass:hf_convex}, MLMC-SGD will have a smaller upper bound on the optimality gap than SGD at a budget $B$ if 
    \begin{align}
        \frac{W}{c_H W_G} &\geq \mathbf{R}_H(x_1) - \mathbf{R}_*, \quad\quad\quad \frac{W_{ML}}{c_H W_{G,ML}} \geq \mathbf{R}_H(x_1) - \mathbf{R}_*
    \end{align}
    and
    \begin{align}
        0 &< \frac{ W}{ 2 c_H^{-1} L_H W_{G} - 1 + B } - \frac{\theta W_{ML}}{ 2\theta c_H^{-1} L_H W_{G,ML} - \theta + B }
        \label{eq:MLMC_improve_cond}
    \end{align}
    where $\theta \equiv C_{ML}/C$ is the cost per iteration ratio for MLMC-SGD and SGD.
\end{theorem}
\begin{proof}
    When multiplied by $\frac{2L_H W_G}{c_H}$, the required conditions become
    \begin{align}
        \frac{2L_H W}{c_H^2} \geq \frac{2L_H W_G}{c_H}(\mathbf{R}_H(x_1) - \mathbf{R}_*) ~~~\textrm{and}~~ \frac{2L_H W_{ML}}{c_H^2} \geq \frac{2L_H W_{G,ML}}{c_H}(\mathbf{R}_H(x_1) - \mathbf{R}_*).
    \end{align}
    From Equation \eqref{eq:opt_conv_const}, these conditions imply the constants of convergence for SGD and MLMC-SGD are $\nu=\frac{2L_HW}{c_H^2}$ and $\nu_{ML} = \frac{2L_HW_{ML}}{c_H^2}$. 
    Now, let the cost of a single iteration of MLMC-SGD be greater than the cost of a single SGD iteration $C_{ML} \geq C$, and let the ratio of per-iteration costs be
    $\theta = C_{ML}/C$.
    The current budget of each algorithm is denoted as $B$ (in units of a single high-fidelity gradient cost).
    Using the definitions of $\nu, \gamma, \nu_{ML}, \gamma_{ML}$,
    \begin{align}
        \frac{\theta W_{ML}}{ 2\theta c^{-1} L_H W_{G,ML} - \theta + B } &< \frac{ W}{ 2 c_H^{-1} L_H W_{G} - 1 + B } \implies
        \frac{\theta\nu_{ML}}{\theta\gamma_{ML} +B} < \frac{\nu}{\gamma +B} \label{eq:sgd_mlmc_bound}.
    \end{align}
    
    We compare SGD and MLMC-SGD at a given budget $B$.
    Let $k_{SGD}$ be the number of SGD iterations that equate to the given budget $k_{SGD}=B$.
    Then, if one MLMC-SGD iteration has the same budget as $\theta$ SGD iterations, then the iteration index of MLMC-SGD is
        $k_{ML} = \frac{k_{SGD}}{\theta}= \frac{B}{\theta}$
    where $k_{ML}$ is the equivalent-cost iteration index of MLMC-SGD.
    Then,
    \begin{align}
            \mathbf{R}_H(w_{k_{ML}}) - \mathbf{R}_* \leq \frac{\nu_{ML}}{\gamma_{ML} + k_{ML}} = \frac{\theta\nu_{ML}}{\theta\gamma_{ML} +B} ~~~~\textrm{and}~~~~
            \mathbf{R}_H(w_{k_{SGD}}) - \mathbf{R}_* \leq \frac{\nu}{\gamma +B}. \nonumber
    \end{align}
    Therefore, from Equation \eqref{eq:sgd_mlmc_bound}, MLMC-SGD has a smaller upper bound on the expected optimality gap than SGD with a budget $B$.
\end{proof}

We now provide proof of Theorem \ref{theo:MLMC_faster} to show that MLMC-SGD has a smaller upper bound on the optimality gap than SGD asymptotically.
\begin{proof}[Proof of Theorem \ref{theo:MLMC_faster}]
    Bottou et al. \cite{bottou2018optimization} states that 
    \begin{align}
        \frac{W}{c_H W_G} &\geq \mathbf{R}_H(x_1) - \mathbf{R}_* \quad\quad\textrm{and}\quad\quad \frac{W_{ML}}{c_H W_{G,ML}} \geq \mathbf{R}_H(x_1) - \mathbf{R}_*
    \end{align}
    are true asymptotically as $B\to \infty$. Thus, to use Theorem \ref{theo:mlmc_sgd}, we show that 
    \begin{align}
        0 &< \frac{ W}{ 2 c_H^{-1} L_H W_{G} - 1 + B } - \frac{\theta W_{ML}}{ 2\theta c_H^{-1} L_H W_{G,ML} - \theta + B }
    \end{align}
    is satisfied asymptotically. We find a common denominator
    \begin{align}
        0 &< \frac{ W(2\theta c_H^{-1} L_H W_{G,ML} - \theta + B)}{ (2 c_H^{-1} L_H W_{G} - 1 + B)(2\theta c_H^{-1} L_H W_{G,ML} - \theta + B) } \\
        &\quad\quad- \frac{\theta W_{ML}(2 c_H^{-1} L_H W_{G} - 1 + B)}{ (2\theta c_H^{-1} L_H W_{G,ML} - \theta + B)(2 c_H^{-1} L_H W_{G} - 1 + B) } \nonumber\\
        0 &< \frac{ (W - \theta W_{ML})B + (2\theta c_H^{-1} L_H W_{G,ML} - \theta) - \theta W_{ML}(2 c_H^{-1} L_H W_{G} - 1)}{ (2 c_H^{-1} L_H W_{G} - 1 + B)(2\theta c_H^{-1} L_H W_{G,ML} - \theta + B) }.
    \end{align}
    Therefore, as the budget $B\to\infty$, the above inequality holds when the dominant term in the numerator $(W - \theta W_{ML})>0$. 
    Therefore, asymptotically, MLMC-SGD has a smaller upper bound on the optimality gap than SGD when
        $C_{ML}W_{ML} < CW$
    using Theorem \ref{theo:mlmc_sgd}.
\end{proof}

\section{Appendix - Corrected Function} \label{app:corr_func}

We show that $\mathbf{\hat{R}}_C(x)$ is strongly convex 
with constant $c_L$ under Assumption \ref{ass:lf_smooth_convex}
For all $y,z\in\mathbb{R}^d$ given an $x$
\begin{align}
    \mathbf{\hat{R}}_C(y) &= \mathbf{\hat{R}}^M_L(y) + (\mathbf{\hat{R}}^N_H(x)-\mathbf{\hat{R}}^N_L(x)) + (\mathbf{\hat{G}}^N_H(x)-\mathbf{\hat{G}}_L^N(x))^\top (y-x) \nonumber\\
    &\geq \mathbf{\hat{R}}^M_L(z) + \nabla\mathbf{\hat{R}}_L^{M}(z)^\top (y-z) + \frac{1}{2}\hat{c}_L||y-z||^2 \nonumber
    \\ &\quad\quad 
    + (\mathbf{\hat{R}}^N_H(x)-\mathbf{\hat{R}}^N_L(x)) + (\mathbf{\hat{G}}^N_H(x)-\mathbf{\hat{G}}^N_L(x))^\top (y-x) \\
     &= \mathbf{\hat{R}}_C(z)
     + (\nabla\mathbf{\hat{R}}_L^{M}(z) + \mathbf{\hat{G}}^N_H(x)-\mathbf{\hat{G}}^N_L(x)
     )^\top (y-z) + \frac{1}{2}\hat{c}_L||y-z||^2 \\
     &= \mathbf{\hat{R}}_C(z)
     + \nabla \mathbf{\hat{R}}_C(z) ^\top (y-z) + \frac{1}{2}\hat{c}_L||y-z||^2.
\end{align}
Thus, $\mathbf{\hat{R}}_C(x)$ is convex.

\section{Trust-Phase Convergence}\label{app:trust_conv}

We begin by finding an expression for the step taken by the trust-phase. The trust-phase minimizes $\mathbf{\hat{R}}_C$ corrected at $x_k$. Due to strong convexity of $\mathbf{\hat{R}}_C$ (shown in Appendix \ref{app:corr_func}), there exists a unique global minimum $x$ that satisfies the optimality condition
\begin{align}
    \nabla \mathbf{\hat{R}}_C(x) = \nabla \mathbf{\hat{R}}_L^K(x) - \nabla \mathbf{\hat{R}}_L^K(x_k)+ \nabla \mathbf{\hat{R}}_H^N(x_k)=0.
\end{align}
Let us then define an effective step $s_k = x - x_k.$
By the fundamental theorem of calculus and a change of variables,
\begin{align}
    \nabla \mathbf{\hat{R}}_L^M(x_k+s_k) - \nabla \mathbf{\hat{R}}_L^M(x_k) = \int_0^1 \nabla^2\mathbf{\hat{R}}_L^M(x_k+ts_k)s_k dt.
\end{align}
We plug this expression into the optimality condition to find
\begin{align}
    0&=\int_0^1 \nabla^2\mathbf{\hat{R}}_L^M(x_k+ts_k)s_k dt + \nabla \mathbf{\hat{R}}_H^N(x_k)\\
    s_k &= -\left(\int_0^1 \nabla^2\mathbf{\hat{R}}_L^M(x_k+ts_k) dt\right)^{-1}\nabla \mathbf{\hat{R}}_H^N(x_k).
\end{align}
Let $B_k \equiv\left(\int_0^1 \nabla^2\mathbf{\hat{R}}_L^M(x_k+ts_k) dt\right)^{-1}$ such that $s_k = -B_k\nabla \mathbf{\hat{R}}_H^N(x_k)$.
By Assumption \ref{ass:lf_smooth_convex}, the eigenvalues of $\nabla^2\mathbf{\hat{R}}_L^M(x)$ are bounded above and below by 
$L_L$ and $c_L$ for all $x$.
Thus, for any unit vector $u\in\mathbb{R}^d$,
\begin{align}
    c_L&\leq u^T\nabla^2\mathbf{\hat{R}}_L^M(x_k+ts_k)u \leq L_L \\
    \int_0^1c_L dt&\leq \int_0^1u^T\nabla^2\mathbf{\hat{R}}_L^M(x_k+ts_k)u dt\leq \int_0^1L_L dt \\
    c_L &\leq u^T\left(\int_0^1\nabla^2\mathbf{\hat{R}}_L^M(x_k+ts_k)dt\right)u\leq L_L.
\end{align}
Thus, the eigenvalues of $\int_0^1 \nabla^2\mathbf{\hat{R}}_L^M(x_k+ts_k) dt$ are similarly bounded by $L_L$ and $c_L$.
Finally, the eigenvalues of $B_k$ are upper and lower bounded by $1/c_L$ and $1/L_L$.

Now, let the step taken by the trust-phase be represented as $x_{k+1} = x_k + \lambda s_k$ for some learning rate $\lambda>0$.
By continuous differentiability
of the high-fidelity risk metric,
\begin{align}
    \mathbf{R}_H(x_{k+1}) &\leq \mathbf{R}_H(x_{k}) +\nabla\mathbf{R}_H(x_{k})^T(x_{k+1}-x_k) + \frac{L}{2}||x_{k+1}-x_k||^2 \\
    \mathbf{R}_H(x_{k+1}) &\leq \mathbf{R}_H(x_{k}) +\lambda\nabla\mathbf{R}_H(x_{k})^Ts_k + \frac{L\lambda^2}{2}||s_k||^2 \\
    \mathbf{R}_H(x_{k+1}) &\leq \mathbf{R}_H(x_{k}) -\lambda\nabla\mathbf{R}_H(x_{k})^T B_k\nabla \mathbf{\hat{R}}_H^N(x_k) + \frac{L\lambda^2}{2}||B_k\nabla \mathbf{\hat{R}}_H^N(x_k)||^2. 
\end{align}
We take the conditional expectation given $x_k$ and define $\Delta_{k+1}\equiv\mathbb{E}[\mathbf{R}_H(x_{k+1})]-\mathbf{R}_\star$.
Thus,
\begin{align}
    \Delta_{k+1} &\leq \Delta_{k} -\lambda\nabla\mathbf{R}_H(x_{k})^T \mathbb{E}[B_k\nabla \mathbf{\hat{R}}_H^N(x_k)] + \frac{L\lambda^2}{2}\mathbb{E}[||B_k\nabla \mathbf{\hat{R}}_H^N(x_k)||^2] \\
    &\leq \Delta_{k} -\lambda\nabla\mathbf{R}_H(x_{k})^T \mathbb{E}[B_k]\nabla \mathbf{R}_H(x_k) + \frac{L\lambda^2}{2}\mathbb{E}[||B_k\nabla \mathbf{\hat{R}}_H^N(x_k)||^2]
\end{align}
since $B_k$ and $\nabla \mathbf{\hat{R}}_H^N(x_k)$ are independent given $x_k$ because $B_k$ is evaluated with $M$ independent samples from the $N$-sample estimator in $\nabla \mathbf{\hat{R}}_H^N(x_k)$. 
For any vector $u\in\mathbb{R}^d$, the norm of a matrix-vector product is bounded with the matrix eigenvalues $(1/L_L)||u||\leq||B_ku||\leq (1/c_L)||u||$.
Thus, we use the eigenvalue bounds, and defining $W_G\equiv W_V+1$ to obtain
\begin{align}
    \Delta_{k+1} &\leq \Delta_{k} -\frac{\lambda}{L_L}||\nabla \mathbf{R}_H(x_k)||^2 + \frac{L\lambda^2}{2c_L^2}\mathbb{E}[||\nabla \mathbf{\hat{R}}_H^N(x_k)||^2] \nonumber\\
    \Delta_{k+1} &\leq \Delta_{k} -\frac{\lambda}{L_L}||\nabla \mathbf{R}_H(x_k)||^2 + \frac{L\lambda^2}{2c_L^2}(W+(W_V+1)||\nabla\mathbf{R}_H(x_k)||^2) \nonumber\\
    &\leq \Delta_{k} +(-\frac{\lambda}{L_L}+ \frac{L\lambda^2}{2c_L^2}W_G)||\nabla \mathbf{R}_H(x_k)||^2 + \frac{L\lambda^2}{2c_L^2}W \nonumber\\
    &\label{eq:trust_decay}
\end{align}
by assumption \ref{ass:var_bound}. 
By summing from $k\in\{1,\ldots,P\}$ and using Assumption \ref{ass:obj_bound},
\begin{align}
    \mathbb{E}[\mathbf{R}_H(x_{k+1}) -\mathbf{R}_H(x_{k})] &\leq (-\frac{\lambda}{L_L}+ \frac{L\lambda^2}{2c_L^2}W_G)||\nabla \mathbf{R}_H(x_k)||^2 + \frac{L\lambda^2}{2c_L^2}W \nonumber\\
    \mathbb{E}[\mathbf{R}_H(x_{P+1}) -\mathbf{R}_H(x_{1})] &\leq (-\frac{\lambda}{L_L}+ \frac{L\lambda^2}{2c_L^2}W_G)\sum_{k=1}^P||\nabla \mathbf{R}_H(x_k)||^2 + P\frac{L\lambda^2}{2c_L^2}W \nonumber\\
    \mathbf{R}_\star -\mathbf{R}_H(x_{1}) &\leq (-\frac{\lambda}{L_L}+ \frac{L\lambda^2}{2c_L^2}W_G)\sum_{k=1}^P||\nabla \mathbf{R}_H(x_k)||^2 + P\frac{L\lambda^2}{2c_L^2}W. \nonumber
\end{align}
Thus, by setting $\lambda=\frac{c_L^2}{L L_L W_G}$
\begin{align}
    \mathbf{R}_\star -\mathbf{R}_H(x_{1}) &\leq \left(-\frac{c_L^2}{L_L^2 LW_G}+ \frac{1}{2}\frac{c_L^2}{L_L^2 LW_G}\right)\sum_{k=1}^P||\nabla \mathbf{R}_H(x_k)||^2 + \frac{P}{2}\frac{Wc_L^2}{L_L^2 LW_G^2} \nonumber\\
    \mathbf{R}_\star -\mathbf{R}_H(x_{1}) &\leq \left(-\frac{c_L^2}{2L_L^2 LW_G}\right)\sum_{k=1}^P||\nabla \mathbf{R}_H(x_k)||^2 + \frac{P}{2}\frac{Wc_L^2}{L_L^2 LW_G^2}. \nonumber
\end{align}

Finally, by rearranging and taking the limit
\begin{align}
    \frac{c_L^2}{2L_L^2 LW_G}\sum_{k=1}^P||\nabla \mathbf{R}_H(x_k)||^2 &\leq  \frac{P}{2}\frac{Wc_L^2}{L_L^2 LW_G^2} - (\mathbf{R}_\star -\mathbf{R}_H(x_{1})) \nonumber\\
    \frac{1}{P}\sum_{k=1}^P||\nabla \mathbf{R}_H(x_k)||^2 &\leq  \frac{LW_GWL_L^2c_L^2}{L_L^2c_L^2LW_G^2} - \frac{2LL_L^2W_G(\mathbf{R}_\star -\mathbf{R}_H(x_{1}))}{Pc_L^2} \nonumber\\
    &\leq  \frac{W}{W_G} - \frac{2LL_L^2W_G(\mathbf{R}_\star -\mathbf{R}_H(x_{1}))}{Pc_L^2} \nonumber\\
    \lim_{P\to\infty} \mathbb{E}\left[\frac{1}{P}\sum_{k=1}^P||\nabla \mathbf{R}_H(x_k)||^2\right] &\leq \frac{W}{W_G}. \nonumber
\end{align}

\section{Convergence of trust-phase with convexity}\label{app:trust_convex}

We begin from Equation \eqref{eq:trust_decay} in Appendix \ref{app:trust_conv}.
By letting $\lambda=\frac{c_L^2}{L L_L W_G}$,
\begin{align}
    \Delta_{k+1} &\leq \Delta_{k} +(-\frac{\lambda}{L_L}+ \frac{L\lambda^2}{2c_L^2}W_G)||\nabla \mathbf{R}_H(x_k)||^2 + \frac{L\lambda^2}{2c_L^2}W\\ 
    &\leq \Delta_{k} +(-\frac{c_L^2}{L L_L^2 W_G}+ \frac{1}{2}\frac{ c_L^2}{L L_L^2 W_G})||\nabla \mathbf{R}_H(x_k)||^2 + \frac{c_L^2}{2L L_L^2 W_G^2}W\\ 
    &\leq \Delta_{k} -\frac{ c_L^2}{2L L_L^2 W_G}||\nabla \mathbf{R}_H(x_k)||^2 + \frac{c_L^2}{2L L_L^2 W_G^2}W.
\end{align}
By the property of $c_H$-strong convexity ($2c_H\Delta_k\leq||\nabla \mathbf{R}_H(x_k)||^2$) and Assumption \ref{ass:hf_convex},
\begin{align}
    \Delta_{k+1} &\leq \Delta_{k} -\frac{ c_L^2}{2L L_L^2 W_G}2c_H \Delta_k + \frac{c_L^2}{2L L_L^2 W_G^2}W\\ 
    &\leq \left( 1 -\frac{ c_H c_L^2}{L L_L^2 W_G} \right)\Delta_k + \frac{c_L^2}{2L L_L^2 W_G^2}W.
\end{align}
Now, by subtracting $\frac{W}{2W_G c_H}$ from both sides
\begin{align}
    \left(\Delta_{k+1} - \frac{W}{2W_Gc_H}\right)&\leq \left( 1 -\frac{ c_H c_L^2}{L L_L^2 W_G} \right)\Delta_k + \frac{c_L^2}{2L L_L^2 W_G^2}W - \frac{W}{2W_Gc_H}\\
    \left(\Delta_{k+1} - \frac{W}{2W_Gc_H}\right)&\leq \left( 1 -\frac{ c_H c_L^2}{L L_L^2 W_G} \right)\left(\Delta_k - \frac{W}{2W_Gc_H}\right) \label{eq:final_form}
\end{align}
which is a contraction if $0 < \left( 1 -\frac{ c_H c_L^2}{L L_L^2 W_G} \right) < 1$.
By definition, $L_H\geq c_H$, $L_L\geq c_L$ and $W_G>1$.
Thus,
\begin{align}
    L_H \geq c_H >0 \\
    -\frac{L_H}{c_L^2} W_G < -\frac{L_H}{L_L^2} <0 \\
    -1 < -\frac{c_L^2}{W_G L_L^2} <0 \\
    0 < 1-\frac{c_L^2}{W_G L_L^2} < 1 
\end{align}
we have a contraction.
Finally, applying recursion with \eqref{eq:final_form} we find
\begin{align}
    \lim_{k\to\infty} \Delta_k \leq \frac{W}{2W_G c_H}.
\end{align}

\end{document}